\documentclass[11pt]{amsart}
\usepackage{geometry,color,amsmath,amsfonts,amssymb,amscd,amsthm,amsbsy,upref,comment,pgfplots}
\usepackage{enumitem}
\usepackage[colorlinks=true,linkcolor=black,anchorcolor=black,citecolor=black,filecolor=black,menucolor=black,runcolor=black,urlcolor=black]{hyperref}
\pgfplotsset{compat=1.18}

\numberwithin{equation}{section}

\newcommand{\N}{\mathbb{N}}

\newcommand{\R}{\mathbb{R}}
\newcommand{\C}{\mathbb{C}}

\newcommand{\vp}{\varepsilon}

\newtheorem{thm}{Theorem}[section]
\newtheorem{cor}[thm]{Corollary}
\newtheorem{prop}[thm]{Proposition}

\newtheorem{prob}[thm]{Problem}

\begin{document}

\title{Declipping and the recovery of vectors from saturated measurements}

\author[W. Alharbi]{ Wedad Alharbi }
\author[D. Freeman]{ Daniel Freeman}\author[D. Ghoreishi]{ Dorsa Ghoreishi}\author[B. Johnson]{Brody Johnson}\author[N.~L. Randrianarivony]{ N. Lovasoa Randrianarivony}
\address{Department of Mathematics and Statistics\\
Saint Louis University\\
St. Louis, MO  63103,  USA
} \email{wedad.alharbi@slu.edu}
\email{daniel.freeman@slu.edu}
\email{dorsa.ghoreishi@slu.edu}
\email{brody.johnson@slu.edu}
\email{nirina.randrianarivony@slu.edu}

\begin{abstract}

A frame $(x_j)_{j\in J}$ for a Hilbert space $H$ allows for a linear and stable reconstruction of any vector $x\in H$ from the linear measurements $(\langle x,x_j\rangle)_{j\in J}$. However, there are many situations where some information in the frame coefficients is lost.  In applications where one is using sensors with a fixed dynamic range, any measurement above that range is registered as the maximum, and any measurement below that range is registered as the minimum. Depending on the context, recovering a vector from such measurements is called either declipping or saturation recovery. We initiate a frame theoretic approach to saturation recovery in a similar way to what \cite{BCE06} did for phase retrieval.  We characterize when saturation recovery is possible, show optimal frames for use with saturation recovery correspond to minimal multi-fold packings in projective space, and prove that the classical frame algorithm may be adapted to this non-linear problem to provide a reconstruction algorithm.

\end{abstract}


\thanks{The second and third authors were supported by NSF grant 2154931. }

\maketitle

\section{Introduction}

In audio recording and electrical engineering, clipping occurs whenever the amplitude of a signal exceeds a sensor's measurement limits.  Any measurement above the range is registered as the maximum, while any measurement below the range is registered as the minimum. In other applied situations where multiple sensors are taking discrete measurements, it is often the case that each sensor can only register values up to a certain magnitude, after which it becomes saturated.  See Figures \ref{fig:saturation-continuous} and \ref{fig:saturation-discrete}, respectively, for illustrations of saturation and its effect on continuous and discrete measurements.  Saturation is not a problem as long as one can choose sensors whose effective range allows accurate measurement for every signal of interest.  However, in many situations this is not feasible.  For instance, sensors for digital camera pixels can only measure a certain amount of brightness, and any brightness beyond that is simply output as the maximum.  Typically, the recovery of a vector from continuous measurements which have been corrupted in this way is called declipping, and the recovery of a vector from discrete measurements is called saturation recovery.  
As this problem arises in many different situations, there has been significant research on saturation recovery and declipping in applications.  Typically, methods for doing saturation recovery rely on assumptions about the structure of the signal being reconstructed, such as assumptions about the structure of human speech \cite{HS14,MH19, GKGB21} or the signal being approximable by a $k$-sparse vector as in compressed sensing \cite{LBDB11,FN17,FL18}. 
  The goal of this paper is to introduce a frame theoretic approach to saturation recovery which will provide a general mathematical foundation for understanding when saturation recovery is possible for all vectors in the unit ball of a Hilbert space.  This approach is analogous to what Balan, Casazza, and Edidin did for the use of frame theory in phase retrieval \cite{BCE06}.   


\begin{figure}[htp]
\begin{minipage}[t]{0.49\textwidth}
\vspace{10pt}
\centering
\begin{tikzpicture}
\begin{axis}[width=2.75in, height=2.5in, domain=0:30, xmin=0, xmax=1, xtick={0,0.5,1}, xticklabels={$0$,$\frac{1}{2}$,$1$}, extra x ticks={}, extra x tick labels={}, extra x tick style={xticklabel style={xshift=-0.9ex}}, ytick={-1,-0.55,0,0.55,1}, yticklabels={$-1$,$-\lambda$,0,$\lambda$,1}, ymin=-1.125, ymax=1.125,legend pos=outer north east,legend cell align=left,axis lines=center,axis line style={shorten >=-10pt, shorten <=-10pt}, xlabel={$t$},x label style={at={(current axis.right of origin)},anchor=north, below=3mm,right=5mm},ylabel={$\langle x, x_{t}\rangle$},minor tick num = 1]

\addplot[domain=0:1,samples=201]{0.35*sin(11*2*pi*deg(x))+0.2*sin(2*3*pi*deg(x)+2*pi/3)-0.4*cos(2*5*pi*deg(x)};
\addplot[dashed,domain=0:1,samples=3]{0.55};
\addplot[dashed,domain=0:1,samples=3]{-0.55};

\end{axis}
\end{tikzpicture}    

(a)
\end{minipage} \begin{minipage}[t]{0.49\textwidth}
\vspace{10pt}
\centering
\begin{tikzpicture}
\begin{axis}[width=2.75in, height=2.5in, domain=0:30, xmin=0, xmax=1, xtick={0,0.5,1}, xticklabels={$0$,$\frac{1}{2}$,$1$}, extra x ticks={}, extra x tick labels={}, extra x tick style={xticklabel style={xshift=-0.9ex}}, ytick={-1,-0.55,0,0.55,1}, yticklabels={$-1$,$-\lambda$,0,$\lambda$,1}, ymin=-1.125, ymax=1.125,legend pos=outer north east,legend cell align=left,axis lines=center,axis line style={shorten >=-10pt, shorten <=-10pt}, xlabel={$t$},x label style={at={(current axis.right of origin)},anchor=north, below=3mm,right=5mm},ylabel={$\Phi_{\lambda}(\langle x, x_{t}\rangle)$},minor tick num = 1]

\addplot[domain=0:1,samples=201]{sign(0.35*sin(11*2*pi*deg(x))+0.2*sin(2*3*pi*deg(x)+2*pi/3)-0.4*cos(2*5*pi*deg(x)))*min(0.55,abs(0.35*sin(11*2*pi*deg(x))+0.2*sin(2*3*pi*deg(x)+2*pi/3)-0.4*cos(2*5*pi*deg(x))))};
\addplot[dashed,domain=0:1,samples=3]{0.55};
\addplot[dashed,domain=0:1,samples=3]{-0.55};


\end{axis}
\end{tikzpicture}   

(b)
\end{minipage}

\caption{Saturation of continuous frame coefficients: (a) Unsaturated frame coefficients (b) Saturated frame coefficients.} \label{fig:saturation-continuous}
\end{figure}
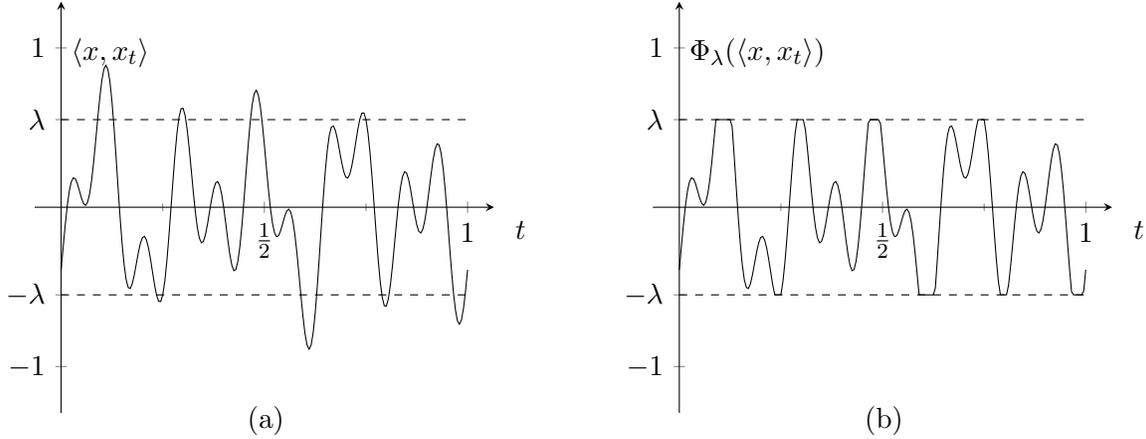

\begin{figure}[htp]
\begin{minipage}[t]{0.49\textwidth}
\vspace{10pt}
\centering
\begin{tikzpicture}
\begin{axis}[width=2.75in, height=2.5in, domain=0:30, xmin=0, xmax=30, xtick={0,5,10,15,20,25,30}, extra x ticks={}, extra x tick labels={}, extra x tick style={xticklabel style={xshift=-0.9ex}}, ytick={-1,-0.55,0,0.55,1}, yticklabels={$-1$,$-\lambda$,0,$\lambda$,1}, ymin=-1.125, ymax=1.125,legend pos=outer north east,legend cell align=left,axis lines=center,axis line style={shorten >=-10pt, shorten <=-10pt}, xlabel={$j$},x label style={at={(current axis.right of origin)},anchor=north, below=3mm,right=5mm},ylabel={$\langle x, x_{j}\rangle$}, minor tick num = 4]

\addplot+[ycomb,color=black,mark size=1.375pt,mark options={fill=black}] coordinates {(1,0.412092) (2,-0.936334) (3,-0.446154) (4,-0.907657) (5,-0.805736) (6,0.646916) (7,0.389657) (8,-0.365801) (9,0.900444) (10,-0.931108) (11,-0.122511) (12,-0.236883) (13,0.531034) (14,0.590400) (15,-0.626255) (16,-0.020471) 
(17,-0.108828) (18,0.292626) (19,0.418730) (20,0.509373) (21,-0.447950) (22,0.359405) (23,0.310196) (24,-0.674777) 
(25,-0.762005) (26,-0.003272) (27,0.919488) (28,-0.319229) (29,0.170536) (30,-0.552376) };

\addplot[dashed,domain=0:30,samples=3]{0.55};
\addplot[dashed,domain=0:30,samples=3]{-0.55};

\end{axis}
\end{tikzpicture}    

(a)
\end{minipage} \begin{minipage}[t]{0.49\textwidth}
\vspace{10pt}
\centering
\begin{tikzpicture}
\begin{axis}[width=2.75in, height=2.5in, domain=0:30, xmin=0, xmax=30, xtick={0,5,10,15,20,25,30}, extra x ticks={}, extra x tick labels={}, extra x tick style={xticklabel style={xshift=-0.9ex}}, ytick={-1,-0.55,0,0.55,1}, yticklabels={$-1$,$-\lambda$,0,$\lambda$,1}, ymin=-1.125, ymax=1.125,legend pos=outer north east,legend cell align=left,axis lines=center,axis line style={shorten >=-10pt, shorten <=-10pt}, xlabel={$j$},x label style={at={(current axis.right of origin)},anchor=north, below=3mm,right=5mm},ylabel={$\phi_{\lambda}(\langle x, x_{j}\rangle)$}, minor tick num = 4]

\addplot+[ycomb,color=black,mark size=1.375pt,mark options={fill=black}] coordinates {(1,0.412092) (2,-0.550000) (3,-0.446154) (4,-0.550000) (5,-0.550000) (6,0.550000) (7,0.389657) (8,-0.365801) 
(9,0.550000) (10,-0.550000) (11,-0.122511) (12,-0.236883) (13,0.531034) (14,0.550000) (15,-0.550000) (16,-0.020471) 
(17,-0.108828) (18,0.292626) (19,0.418730) (20,0.509373) (21,-0.447950) (22,0.359405) (23,0.310196) (24,-0.550000) 
(25,-0.550000) (26,-0.003272) (27,0.550000) (28,-0.319229) (29,0.170536) (30,-0.550000)};

\addplot[dashed,domain=0:30,samples=3]{0.55};
\addplot[dashed,domain=0:30,samples=3]{-0.55};

\end{axis}
\end{tikzpicture}

(b)
\end{minipage}

\caption{Saturation of discrete frame coefficients: (a) Unsaturated frame coefficients (b) Saturated frame coefficients.} \label{fig:saturation-discrete}
\end{figure}
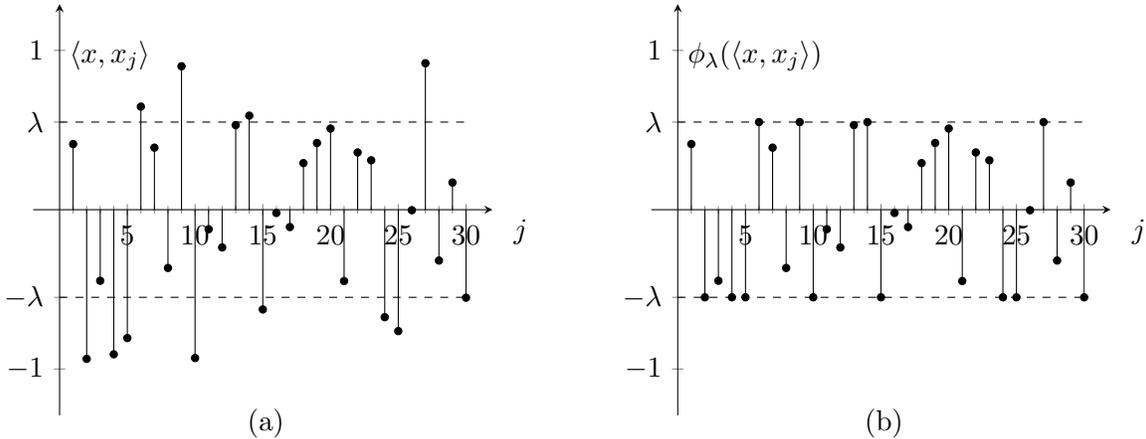

A {\em frame} for a Hilbert space $H$ is a collection of vectors $(x_j)_{j\in J}\subseteq H$  such that there exists universal constants $A,B>0$, called the {\em frame bounds}, which satisfy
\begin{equation}\label{E:framed}
A\|x\|^2\leq \sum_{j\in J}|\langle x,x_j\rangle|^2\leq B\|x\|^2\hspace{1cm}\textrm{ for all }x\in H.
\end{equation}
A frame is called {\em tight } if  \eqref{E:framed} is satisfied with $A=B$, and is called {\em Parseval} if \eqref{E:framed} is satisfied with $A=B=1$.
Note that $(x_j)_{j\in J}$ is a frame if and only if the {\em analysis operator}  $\Theta:H\rightarrow \ell_2(J)$ is an isomorphic embedding where $\Theta(x)=(\langle x,x_j\rangle)_{j\in J}$.  Furthermore, $(x_j)_{j\in J}$ is a Parseval frame if and only if the analysis operator is an isometric embedding.  


For $\lambda>0$, we consider the function $\phi_\lambda:\R\rightarrow [-\lambda,\lambda]$ given by $\phi_\lambda(t)=t$ if $|t|\leq \lambda$ and $\phi_\lambda(t)=\mathrm{sign}(t) \lambda$ if $|t|>\lambda$.  Given any index set $J$, we denote $\Phi_\lambda:\ell_2(J)\rightarrow \ell_2(J)$ to be the map $\Phi_\lambda((a_j)_{j\in J})=(\phi_\lambda(a_j))_{j\in J}$ for all $(a_j)_{j\in J}\in\ell_2(J)$.  Let $(x_j)_{j\in J}$ be a frame for a Hilbert space $H$ with analysis operator $\Theta:H\rightarrow \ell_2(J)$.  We say that $(x_j)_{j\in J}$ does {\em $\lambda$-saturation recovery on a set $X\subseteq H$} if the map $\Phi_\lambda \Theta|_X:X\rightarrow\ell_2(J)$ is one-to-one.  The case of saturation recovery in compressed sensing was considered in \cite{LBDB11,FN17,FL18} where $X$ is the set of $k$-sparse vectors of norm at most some $\alpha>0$.  We will be focusing on the case $X=\alpha B_H$ to be a scalar multiple of the closed unit ball of $H$, denoted $B_{H}$.
In Section \ref{S:SR}, we will characterize exactly when $\lambda$-saturation is possible on a set of the form $\alpha B_H$.  In particular, we prove that $\lambda$-saturation is possible on $\alpha B_{\R^n}$ if and only if at every point $x\in\alpha B_{\R^n}$ the unsaturated coordinates of $(x_j)_{j\in J}$ form a frame of $\R^n$.

In Section \ref{S:Min} we consider the problem of given $m,n\in\N$ what is the least $\lambda>0$ so that there exists a frame of $m$ unit vectors for $\R^n$ which does $\lambda$-saturation recovery on $B_{\R^n}$?  We show that this is equivalent to a multi-packing problem in real projective space.  The overlap in frame theory with packing problems in real projective space is an exciting and active area of research \cite{BK20,FIJK21,FJM18,FJM20,FJMP21,IJM20}.  However, this appears to be the first connection between frame theory and multi-packing problems in real projective space.  We hope that our results inspire a new avenue of inquiry in the area, and we pose multiple open problems.

We think of $(x_j)_{j\in J}$ doing $\lambda$-saturation recovery on a set $X\subseteq H$ as it being possible to recover any $x\in X$ from the values $\Phi_\lambda \Theta(x)=(\phi_\lambda(\langle x,x_j\rangle))_{j\in J}$.  However, it is important not only to be able to recover $x$, but that this recovery be stable as well.  We consider the problem of calculating the stability of saturation recovery in Section \ref{S:Stab}.  We say that $(x_j)_{j\in J}$ does $C$-stable $\lambda$-saturation recovery on $X$ for some $C>0$ if the recovery map $\Phi_\lambda \Theta(x)\mapsto x$ is $C$-Lipschitz on the range of $\Phi_\lambda \Theta|_{X}$.  That is, $(x_j)_{j\in J}$ does {\em $C$-stable $\lambda$-saturation recovery on $X$} if 
\begin{equation}\label{E:SSR}
    \|x-y\|\leq C\|\Phi_\lambda\Theta x-\Phi_\lambda\Theta y\|\hspace{1cm}\textrm{ for all }x,y\in X.
\end{equation}
We prove  whenever $(x_j)_{j\in J}$ is a finite frame of $\R^n$ that there exists a critical value $\lambda_c>0$ so that $(x_j)_{j\in J}$ does stable $\lambda$-saturation recovery on $B_{\R^n}$ for every $\lambda> \lambda_c$.  However,  it remains an open problem to determine if $(x_j)_{j\in J}$ always does stable $\lambda_c$-saturation recovery on $B_{\R^n}$.

The frame algorithm is an efficient method for recovering a vector $x\in H$ from the sequence of frame coefficients $(\langle x,x_j\rangle)_{j\in J}$, and it was presented in the same paper of Duffin and Schaeffer which first defined frames \cite{DS52}.  Gr\"ochenig later showed that the frame algorithm could be modified to guarantee even faster performance \cite{G93}.
In Section \ref{S:FA}, we give an adaptation of the frame algorithm which is able to recover any $x\in B_{H}$ from the values $(\phi_\lambda(\langle x,x_j\rangle))_{j\in J}$.  Naturally, the stability of $\lambda$-saturation recovery is directly related to the performance of any algorithm used to implement it.  We prove in Theorem \ref{T:FA_P} that if $(x_j)_{j\in J}$ is a Parseval frame for a Hilbert space $H$ which does $C$-stable $\lambda$-saturation recovery on $2\|x\|B_H$, then the sequence of approximations $(y_k)_{k=0}^\infty$ provided by the $\lambda$-saturated frame algorithm satisfies $\|x-y_k\|\leq (1-C^{-2})^{k/2} \|x\|$ for all $k\in\N$.  The $\lambda$-saturated frame algorithm that we introduce gives a non-linear formula for constructing a sequence of approximations to $x$.  Instead of using this non-linear adaptation, one could just use the linear frame algorithm on the unsaturated coordinates of $(x_j)_{j\in J}$.  We show in Theorem \ref{T:satfa} that, except in trivial circumstances, the guaranteed convergence rate for the $\lambda$-saturated frame algorithm will always outperform the guaranteed convergence rate for doing the linear frame algorithm on the unsaturated coordinates.  In Section \ref{S:exp}, we present the results of some numerical experiments illustrating this fact.

We note that the problem of recovering $x$ from the values $(\phi_\lambda(\langle x,x_j\rangle))_{j\in J}$ is one case of a collection of non-linear inverse problems.  For example, phase retrieval is the problem of recovering $x$ (up to a uni-modular scalar) from the values $(|\langle x,x_j\rangle|)_{j\in J}$ \cite{BCE06}.  The problem of inverting layers of neural nets is considered in \cite{HEB23}, where the goal is to recover $x$ from $(ReLU_\alpha(\langle x,x_j\rangle))_{j\in J} $.  Folding takes a different approach to the problem of measurement saturation by introducing a new nonlinearity before sensor measurements are taken, and this results in the problem of recovering $x$ from the values $(\mathrm{Mod}_{2\lambda}(\langle x,x_j\rangle +\lambda)-\lambda)_{j\in J}$ \cite{BKR20,FB22,FKLB22}.  Each of these non-linear functions leads to separate theorems and methods, and we identify throughout the paper some similarities and differences when presenting our results.

During the course of working on this project we were fortunate to engage in very helpful and interesting conversations with multiple mathematicians on this subject.  We give warm thanks in particular to Akram Aldroubi, Peter Balazs,  Martin Ehler, Daniel Haider, Ilya Krishtal, and Alex Powell.

\section{Saturation recovery}\label{S:SR}

Let $(x_j)_{j\in J}$ be a frame of a finite-dimensional Hilbert space $H$ and let $\lambda,\alpha>0$.  Note that $(x_j)_{j\in J}$ does $\lambda$-saturation recovery on the set $\alpha B_H$ if and only $(x_j)_{j\in J}$ does $C\lambda$-saturation recovery on the set $C\alpha B_H$ for every scalar $C>0$.  This allows us to simplify many of the results in this section by considering the case $\alpha=1$.

For $x\in H$, we denote the coordinates corresponding to the unsaturated frame coefficients of $x$ by $J_{\lambda}(x)=\{j\in J: |\langle x, x_j\rangle|\leq\lambda\}$.  Likewise, we denote $J_{\lambda}^{\sharp}(x)=\{j\in J: |\langle x, x_j\rangle|<\lambda\}$.  Observe that the upper frame bound of $(x_j)_{j\in J}$ limits the cardinality of $(J_{\lambda}(x))^{c}$, so this set must always be finite.  We now show that a frame does $\lambda$-saturation recovery if and only if the non-saturated frame coordinates form a frame for every vector in the ball.  The analogous result in the case of using the ReLU function instead of the saturation function $\phi_\lambda$ is proven in \cite{HEB23}.

\begin{thm}\label{T:1}
Let $(x_j)_{j\in J}$ be a frame for a finite dimensional Hilbert space $H$ and let $\lambda,\alpha>0$.  Then the following are equivalent.
\begin{enumerate}[label=\textup{(\alph*)}]
    \item $(x_j)_{j\in J}$ does $\lambda$-saturation recovery on $\alpha B_H$. \label{T:1a}
    \item For all $x\in H$ with $\|x\|<\alpha$, $(x_j)_{j\in J_{\lambda}^{\sharp}(x)}$ is a frame of $H$. \label{T:1b}
    \item For all $x\in H$ with $\|x\|\leq \alpha$, $(x_j)_{j\in J_{\lambda}(x)}$ is a frame of $H$. \label{T:1c}
\end{enumerate} 
\end{thm}
\begin{proof}
By scaling $\alpha$ and $\lambda$ by $1/\alpha$, we may assume without loss of generality that $\alpha=1$.
We first prove $ \ref{T:1a} \Rightarrow \ref{T:1c} $. 
 Suppose $x\in H$ with $0<\|x\|\leq 1$ and $ \mathrm{span} (x_j)_{j\in J_{\lambda}(x)} \ne H$. Now, we have $|\langle x, x_j\rangle| > \lambda\ $ for all $ j \in (J_{\lambda}(x))^{c}$. As $(J_{\lambda}(x))^c$ must be finite, there exists $0<r<1$ so that  $|\langle rx, x_j\rangle| > \lambda\ $ for all $ j \in (J_{\lambda}(x))^{c}$.   
 
 Let $z\in \mathrm{span} ( (x_j)_{j\in J_{\lambda}(x)})^\perp $ with $0<\|z\|<1-r\|x\|$ and $\|z\|\leq \min_{j\in (J_{\lambda}(x))^{c}}(|\langle rx, x_j\rangle|-\lambda\ )\|x_j\|^{-1}$. 
Let  $j\in (J_{\lambda}(x))^{c}$ with  $\langle x, x_j\rangle \ge \lambda\ $.  We have that 
 \begin{align*}
 \langle rx+z, x_j\rangle&= \langle rx, x_j\rangle+\langle z, x_j\rangle\\
 &\ge \langle rx, x_j\rangle-|\langle z, x_j\rangle|\\
 &\ge \langle rx, x_j\rangle -\|z\| \|x_j\|\\
 &\ge \langle rx, x_j\rangle- (|\langle rx, x_j\rangle|-\lambda\ )\\
 &= \lambda\ .
 \end{align*}
 Thus, we have that $\phi_\lambda \langle rx+z, x_j\rangle=\phi_\lambda \langle rx, x_j\rangle = \lambda$ for $j\in (J_{\lambda}(x))^{c}$ with  $\langle x, x_j\rangle \ge \lambda\ $.  Likewise, we have that $\phi_\lambda \langle rx+z, x_j\rangle=\phi_\lambda \langle rx, x_j\rangle = -\lambda$ for $j\in (J_{\lambda}(x))^{c}$ with $\langle x, x_j\rangle \leq -\lambda\ $.    For $j\in J_{\lambda}(x)$ we have that $\langle z, x_j\rangle=0$ and hence $\langle rx+z, x_j\rangle=\langle rx, x_j\rangle$.  Thus,  $\phi_\lambda \langle rx+z, x_j\rangle=\phi_\lambda \langle rx, x_j\rangle$ for all ${j\in J}$.  Hence $(x_j)_{j\in J}$ does not do $\lambda$-saturation recovery as $\|rx+z\|<1$ and $\|rx\|<1$.

We now prove $\ref{T:1c}\Rightarrow \ref{T:1b}$.  We assume \ref{T:1c} holds and let $x\in B_H$ with $\|x\|<1$ and $x\neq 0$.  We have that $(x_j)_{j\in J_{\lambda}(\|x\|^{-1} x)}$ is a frame of $H$ and $J_{\lambda}(\|x\|^{-1} x)\subseteq J_{\lambda}^{\sharp}(x)$.  Thus, $(x_j)_{j\in J_{\lambda}^{\sharp}(x)}$ is a frame of $H$.

Lastly, we prove $\ref{T:1b} \Rightarrow \ref{T:1a}$.  We assume  that $ (x_{j})_{j\in J_{\lambda}^{\sharp}(x)} $ is frame of $ H $ for all $x\in H$ with $\|x\|<1$. 
Now let $ x, y \in H$ with $\|x\|\leq1$ and $\|y\|\leq 1$ such that  $ \Phi_\lambda T(x)=\Phi_\lambda T(y)$. 

We first consider the case that either $\|x\|< 1$ or $ \|y\| < 1$.  Without loss of generality, we assume that $\|x\|<1$.  Thus, $(x_{j})_{j\in J_{\lambda}^{\sharp}(x)} $ is frame of $ H $ and $ \phi_\lambda  \langle x , x_j \rangle = \phi_\lambda \langle y , x_j \rangle $  for all $j \in J $.  Thus, $\langle x , x_j \rangle =\langle y , x_j \rangle$ for all $j\in J_{\lambda}^{\sharp}(x)$ and hence  $ x = y $ as $(x_j)_{j\in J_{\lambda}^{\sharp}(x)}$ is a frame of $ H $.   
 
 We now consider the case that $\|x\|=\|y\|=1$ and $x\neq y$.  As $x\neq y$, we have that $\|(x+y)/2\|<1$.  But, $\Phi_\lambda T((x+y)/2)=\Phi_\lambda Tx=\Phi_\lambda Ty$.  Indeed, if $j\in J$ is such that $\langle x, x_j\rangle\geq \lambda$ then $\langle y,x_j\rangle\geq\lambda$ and hence
 $$\langle(x+y)/2,x_j\rangle=\frac{1}{2}(\langle x,x_j\rangle+\langle y, x_j\rangle)\geq\frac{1}{2}(\lambda+\lambda)=\lambda.
 $$
 Likewise, if $j\in J$ is such that $\langle x,x_j\rangle\leq-\lambda$ then $\langle(x+y)/2,x_j\rangle\leq -\lambda$.  If $j\in J_{\lambda}^{\sharp}(x)=J_{\lambda}^{\sharp}(y)$ then $\langle x,x_j\rangle=\langle y,x_j\rangle$ and hence $\langle x,x_j\rangle= \langle(x+y)/2,x_j\rangle$.  This proves that $\Phi_\lambda T((x+y)/2)=\Phi_\lambda Tx=\Phi_\lambda T_y$.  As $\|(x+y)/2\|<1$ we have reduced the problem to the previous case and hence $(x+y)/2=x=y$, a contradiction of the assumption that $x\neq y$.
\end{proof}

Let $(x_j)_{j\in J}$ be a finite frame for a real Hilbert space $H$.  Note that if $\lambda\geq \|x_j\|$ for all $j\in J$ then $(x_j)_{j\in J}$ will trivially do $\lambda$-saturation recovery on $B_H$ as none of the frame coefficients for $x\in B_H$ will be saturated.  On the other hand, $(x_j)_{j\in J}$ will not do $\lambda$-saturation recovery on $B_H$ if $\lambda$ is very close to $0$.   We denote $\lambda_c$ to be the infimum of all $\lambda>0$ so that $(x_j)_{j\in J}$ does $\lambda$-saturation recovery on $B_H$.  We will refer to $\lambda_c$ as the {\em critical saturation level}.  For example, if $(x_j)_{j=1}^n$ is an orthonormal basis of $\R^n$ then clearly $\lambda_c=1$, and we will show in Corollary \ref{C:MB} that if $(x_j)_{j=1}^{n+1}\subseteq S_{\R^n}$ is an equiangular  frame of $\R^n$ then $\lambda_c=2^{-1/2}(1+n^{-1})^{1/2}$.

The following theorem shows that when  $(x_j)_{j\in J}$ is a finite frame, we have $\lambda_c>0$ and $(x_j)_{j\in J}$ does $\lambda_c$-saturation recovery on $B_H$. 

\begin{thm}\label{T:no}
Let $(x_j)_{j\in J}$ be a finite frame of a finite dimensional Hilbert space $H$. The following all hold.
\begin{enumerate}[label=\textup{(\alph*)}]
\item $(x_j)_{j\in J}$ does $\lambda_c$-saturation recovery. \label{T:noa}
\item If $\lambda>\lambda_c$ then for all $x\in H$ with $\|x\|\leq 1$ we have that $(x_j)_{j\in J_{\lambda}^{\sharp}(x)}$ is a frame of $H$. \label{T:nob}
\item There exists $x\in H$ with $\|x\|= 1$ so that $(x_j)_{j\in J_{\lambda_c}^{\sharp}(x)}$ is not a frame of $H$. \label{T:noc}
\item There are only finitely many $x\in H$ with $\|x\|\leq 1$ so that $(x_j)_{j\in J_{\lambda_c}^{\sharp}(x)}$ is not a frame of $H$. \label{T:nod}
\end{enumerate}
\end{thm}

\begin{proof}
 We first prove that $(x_j)_{j\in J}$ does $\lambda_c$-saturation recovery by proving that $(x_j)_{j\in J_{\lambda}(x)}$ is a frame for all $x\in B_H$.  Let $x\in H$ with $\|x\|\leq 1$.  We have that
$$ J_{\lambda_c}(x) = \bigcap_{\lambda>\lambda_c} J_{\lambda}(x).$$
As this is an intersection of finite sets, we have that $J_{\lambda_c}(x)=J_{\lambda}(x)$ for some $\lambda>\lambda_c$.  As $(x_j)_{j\in J}$ does $\lambda$-saturation recovery we have that $(x_j)_{j\in J_{\lambda}(x)}$ is a frame of $H$ and hence  $(x_j)_{j\in J_{\lambda_c}(x)}$ is a frame of $H$.  This proves that $(x_j)_{j\in J}$ does $\lambda_c$-saturation recovery which proves (a).

We have for all $\lambda>\lambda_c$ that $J_{\lambda}^{\sharp}(x)\supseteq J_{\lambda_c}(x)$. By (a), for all $x\in H$ with $\|x\|\leq 1$ we have that $(x_j)_{j\in J_{\lambda_c}(x)}$ is a frame of $H$. Hence $(x_j)_{j\in J_{\lambda}^{\sharp}(x)}$ is a frame of $H$ as well.  This proves (b).

 We denote $H_{\lambda_c}$ to be the set of all $x\in H$ such that $(x_j)_{j\in J_{\lambda_c}^{\sharp}(x)}$ is a frame of $H$.  As $(x_j)_{j\in J}$ does $\lambda_c$-saturation recovery we have by Theorem \ref{T:1}\ref{T:1b} that the open unit ball of $H$ is contained in $H_{\lambda_c}$.  For the sake of contradiction, we assume that $B_H\subseteq H_{\lambda_c}$.  Recall that $B_H$ represents the closed unit ball of $H$.  As $H_{\lambda_c}$ is an open set and $B_H$ is compact there exists $\alpha>1$ so that $\alpha B_H\subseteq H_{\lambda_c}$.  Thus, $B_H\subseteq H_{\lambda_c\alpha^{-1}}$ and $(x_j)_{j\in J}$ does $\lambda_c\alpha^{-1}$-saturation recovery.  This contradicts the fact that $\lambda_c$ is the infimum of all $\lambda>0$ so that $(x_j)_{j\in J}$ does $\lambda$-saturation recovery. Hence, $B_H\not\subseteq H_{\lambda_c}$ and there exists $x\in H$ with $\|x\|=1$ so that $(x_j)_{j\in J_{\lambda_c}^{\sharp}(x)}$ is not a frame of $H$, proving (c).

 We claim that for all  $(\vp_j)_{j\in J}\in\{-1,0,1\}^J$ there exists at most one $x\in B_H$ so that $(x_j)_{j\in J_{\lambda_c}^{\sharp}(x)}$ is not a frame of $H$, $\vp_j=\mathrm{sign}(\langle x,x_j\rangle)$ for all $j\in (J_{\lambda_c}^{\sharp}(x))^{c}$, and $\vp_j=0$ for all $j\in J_{\lambda_c}^{\sharp}(x)$.  Assuming the claim, we have that there exists at most $3^{|J|}$ vectors $x\in B_H$ so that $(x_j)_{j\in J_{\lambda_c}^{\sharp}(x)}$ is not a frame of $H$, which proves (d).  
 We now prove that the claim holds.
 Let  $x\in B_H$ so that $(x_j)_{j\in J_{\lambda_c}^{\sharp}(x)}$ is not a frame of $H$.  For the sake of contradiction, we assume that there exists $y\in B_H$ with $y\neq x$, $J_{\lambda_c}^{\sharp}(y)=J_{\lambda_c}^{\sharp}(x)$, and $\mathrm{sign}(\langle y,x_j\rangle)=\mathrm{sign}(\langle x,x_j\rangle)$ for all $j\in (J_{\lambda_c}^{\sharp}(x))^{c}$.  We have that $\|\frac{1}{2}(x+y)\|<1$ and $J_{\lambda_c}^{\sharp}(\frac{1}{2}(x+y))=J_{\lambda_c}^{\sharp}(y)=J_{\lambda_c}^{\sharp}(x)$.  Thus, $(x_j)_{j\in J_{\lambda_c}^{\sharp}(\frac{1}{2}(x+y))}$ is not a frame of $H$.  As $\frac{1}{2}(x+y)$ is contained in the open unit ball of $H$, this contradicts that  $(x_j)_{j\in J}$ does $\lambda_c$-saturation recovery and proves our claim.  Hence, there exists at most $3^{|J|}$ vectors $x\in B_H$ so that $(x_j)_{j\in J_{\lambda_c}^{\sharp}(x)}$ is not a frame of $H$.
\end{proof}

\section{Minimal structure for saturation recovery}\label{S:Min}

Both in phase retrieval and saturation recovery, it is necessary to use a redundant frame rather than a basis.  The redundancy of a frame is necessary to perform these recovery operations, but it is important to know just how much redundancy is needed.  That is, given a fixed Hilbert space $H$, how many vectors should be included in the frame for it to do phase retrieval or $\lambda$-saturation recovery on $B_H$?
For real phase retrieval, it is known that a generic frame of $2n-1$ vectors does phase retrieval for $\R^n$ and that no frame of $2n-2$ vectors can do phase retrieval for $\R^n$ \cite{BCE06}.  For complex phase retrieval, it is known that a generic frame of $4n-4$ vectors does phase retrieval for $\C^n$ for $n\geq2$ and that this is optimal for certain dimensions of the form $n=2^k+1$ \cite{CEHV15}.  However, Cynthia Vinzant constructed a frame of 11 vectors which does phase retrieval for $\C^4$ \cite{V15}, which earned a prize of one coke from Dustin Mixon \cite{M13}.  The problem of determining the minimal number of vectors required to do phase retrieval in $\C^n$ remains open in general.  
The number of unit vectors required to do $\lambda$-saturation recovery on the ball $B_{\R^n}$ depends on both the value $\lambda>0$ and the dimension $n\in\N$.  This naturally leads to the following problems.

\begin{prob}\label{Q:1}
 Let $n\in\N$ and $\lambda>0$.  What is the smallest $m\geq n$ so that there exists a frame of $m$ unit vectors which does $\lambda$-saturation recovery on $B_{\R^n}$?
\end{prob}
\begin{prob}\label{Q:2}
Let $m\geq n$.  What is the smallest $\lambda>0$ so that there exists a frame of $m$ unit vectors which does $\lambda$-saturation recovery on $B_{\R^n}$?
\end{prob}

 Note that by Theorem \ref{T:1}, we have that a frame $(x_j)_{j\in J}$ does $\lambda$-saturation recovery on $B_{\R^n}$ if and only if $(x_j)_{j\in J_{\lambda}(x)}$ is a frame of $\R^n$ for all $x\in B_{\R^n}$.  We will show that this naturally leads to a multiple packing problem in real projective space.  

Let $(X,d)$ be a metric space.  For $x_0\in X$ and $\vp>0$ we denote $B^o_\vp(x_0)=\{x\in X:d(x,x_0)<\vp\}$ to be the open ball of radius $\vp$ centered at $x_0$.  A {\em packing} in $X$ is a collection of disjoint balls $(B^o_\vp(x_j))_{j=1}^m$.  
Given $L\in\N$, we call $(B^o_\vp(x_j))_{j=1}^m$ an {\em $L$-fold multi packing} or {\em $L$-fold multiple packing} if for all $x\in X$ we have that $d(x,x_j)>\vp$ for at most $L$ values of $j\in{1,...,m}$.  Note that $(B^o_\vp(x_j))_{j=1}^m$ is a packing in $X$ if and only if it is a $1$-multi-packing.  When working in Hamming space, packings correspond to error correcting codes as if a code is transmitted with less errors than the radius of the packing then the closest point in the packing will be unchanged \cite{MS77}.  In this setting, multi-packings correspond to list decodable codes where if a code is transmitted with less errors than the radius of the multi-packing, it will be possible to choose a code from a small list of acceptable values \cite{E91}.

We will specifically be considering packings and multi-packings in real projective space.
For $n\geq2$, we consider the $(n-1)$-dimensional projective space, $P^{n-1}$, to be the set of all lines in $\R^n$ which contain the origin.  The distance $d(L_1,L_2)$ between two lines $L_1,L_2\in P^{n-1}$ is defined as the angle $\theta\in[0,\pi/2]$ between the lines.  This definition is chosen as it is consistent with the induced Riemannian metric on $P^{n-1}$, although this will not be necessary for our purposes. For $x\in \R^n\setminus\{0\}$, we denote the span of $x$ by $[x]\in P^{n-1}$.  It follows that if $x,y\in S_{\R^n}$ are unit vectors then $d([x],[y])=\cos^{-1}(|\langle x,y\rangle|)$.

Given a collection of $m$ unit vectors $(x_j)_{j=1}^m$ in $\R^n$, the {\em Welch bound}, \cite{W74}, states that
\begin{equation}\label{E:Welch}
    \max_{j\neq k}|\langle x_j,x_k\rangle|^2\geq \frac{m-n}{n(m-1)}.
\end{equation}
Thus, \eqref{E:Welch} can be used to give an upper bound on the possible $\vp>0$ so that there exists a packing in $P^{n-1}$ of balls of radius $\vp$.  In \cite{SH03} it is proven that a collection of unit vectors $(x_j)_{j=1}^m\subseteq S_{\R^n}$ satisfies equality in \eqref{E:Welch} if and only if $(x_j)_{j=1}^m$ is an equiangular tight frame.  Here, $(x_j)_{j=1}^m\subseteq S_{\R^n}$ is called {\em equiangular} if there is a constant $\alpha$ so that $|\langle x_i,x_j\rangle|=\alpha$ for all $i\neq j$.
Thus, equiangular tight frames provide optimal packings in projective space, but the converse does not hold as there does not exist an equiangular tight frame for $\R^n$ of $m$ vectors for every $m\geq n$.  It is a major open problem in frame theory to determine for which values $m\geq n$ there exists an equiangular tight frame of $m$ vectors for $\R^n$.  However, it is fairly simple to construct an equiangular tight frame of $n+1$ vectors for $\R^n$. Indeed, 
we denote $(e_j)_{j=1}^{n+1}$ to be the unit vector basis for $\R^{n+1}$ and denote $P_{(\mathrm{span}(\sum e_j))^\perp}$ to be the orthogonal projection onto the orthogonal complement of $\sum_{j=1}^{n+1} e_j$.  We let $x_j=P_{(\mathrm{span}(\sum e_j))^\perp} e_j$, and then by symmetry we have that $(x_j/\|x_j\|)_{j=1}^{n+1}$ is an equiangular tight frame for $(\mathrm{span}(\sum e_j))^\perp$ consisting of $n+1$ unit vectors.


We now show that constructing frames of unit vectors which do $\lambda$-saturation recovery on $B_{\R^n}$ is equivalent to a multi-packing problem in real projective space.

\begin{prop}\label{P:pack}
Let $(x_j)_{j=1}^m$ be a full spark frame of unit vectors for $\R^n$ and let $0<\lambda<1$.  Then $(x_j)_{j=1}^m$ does $\lambda$-saturation recovery on $B_{\R^n}$ if and only if $(B^o_\vp([x_j]))_{j=1}^m$ is a $(m-n)$-fold multi packing in $P^{n-1}$ for $\vp=\cos^{-1}(\lambda)$.
\end{prop}
\begin{proof}
    By Theorem \ref{T:1} we have that $(x_j)_{j=1}^m$ does $\lambda$-saturation recovery on $B_{\R^n}$ if and only if $(x_j)_{j\in J_{\lambda}(x)}$ is a frame of $\R^n$ for all $x\in B_{\R^n}$.  As $(x_j)_{j=1}^m$ is full spark, we have for each $x\in B_X$ that $(x_j)_{j\in J_{\lambda}(x)}$ is a frame of $\R^n$ if and only if $|J_{\lambda}(x)|\geq n$.  For $x\in S_{\R^n}$, we have that $j\in J_{\lambda}(x)$ if and only if 
the distance in the projective plane between $[x]$ and $[x_j]$ satisfies 
$$d([x],[x_j])=\cos^{-1}(|\langle x,x_j\rangle|)\geq\cos^{-1}(\lambda).
 $$
Thus, for $\vp=\cos^{-1}(\lambda)$ we have for each $x\in S_{\R^n}$ that $| J_{\lambda}(x)|\geq n$ if and only if $d([x],[x_j])< \vp$ for at most $m-n$ values of $j$.  This gives that $(x_j)_{j=1}^m$ does $\lambda$-saturation recovery on $B_{\R^n}$ if and only if $(B_\vp([x_j]))_{j=1}^m$ is a $(m-n)$-multi cover of $P^{n-1}$. 
\end{proof}

This motivates the following problems which generalize Problems \ref{Q:1} and \ref{Q:2} for the case that $L=m-n$ and $(x_j)_{j=1}^m$ is a full spark frame.

\begin{prob}\label{Q:4} 
We pose the following multi-packing problems in projective space.
\begin{enumerate}[label=\textup{(\arabic*)}]
    \item Let $m,n,L\in\N$, what is the greatest $\vp>0$ so that there exists an $L$-fold multi packing $(B_\vp([x_j]))_{j=1}^m$ in $P^{n-1}$?
    \item Let $m,n\in\N$ and $\vp>0$, what is the least $L\in\N$ so that there exists an $L$-fold multi packing $(B_\vp([x_j]))_{j=1}^m$ in $P^{n-1}$?
    \item Let $m,L\in\N$ and $\vp>0$, what is the least $n\in\N$ so that there exists an $L$-fold multi packing $(B_\vp([x_j]))_{j=1}^m$ in $P^{n-1}$?
    \item Let $n,L\in\N$ and $\vp>0$, what is the greatest $m\in\N$ so that there exists an $L$-fold multi packing $(B_\vp([x_j]))_{j=1}^m$ in $P^{n-1}$?
\end{enumerate}
\end{prob}

Note that Proposition \ref{P:pack} only applies to frames of unit vectors.  In general, a full spark frame $(x_j)_{j=1}^m\subseteq \R^n$ does $\lambda$-saturation recovery on $B_{\R^n}$ if and only if for all $x\in S_{\R^n}$ at most $m-n$ of the frame coefficients have magnitude greater than $\lambda$.  Thus, the critical value for $(x_j)_{j=1}^{m}$ to do $\lambda_c$-saturation recovery on $B_{\R^n}$ is given by
\begin{equation}
    \lambda_c=\max_{x\in S_{\R^n}}\;\max_{1 \leq k_0 < \cdots < k_{m-n} \leq m}\;\min_{0 \leq j \leq m-n} |\langle x, x_{k_j}  \rangle |.
\end{equation}

A metric space $(M,d)$ is said to satisfy the {\em midpoint property} if for every $x,y\in M$ there exists $z\in M$ with $d(x,z)=d(y,z)=d(x,y)/2$.  It follows by a simple triangle inequality argument that if $(M,d)$ satisfies the midpoint property then the greatest $\vp>0$ such that $(B^o_{\vp}(x_j))_{j=1}^m$ is a packing in $M$ is given by $\vp=\min_{i\neq j}d(x_i,x_j)/2$.  Note that $P^{n-1}$ satisfies the midpoint property as when $L_1,L_2\in P^{n-1}$, the midpoint of $L_1$ and $L_2$ is the line bisecting the acute angle between $L_1$ and $L_2$.  Thus,  if $(x_j)_{j=1}^m\subseteq S_{\R^n}$ then the maximal $\vp>0$ such that $(B^o_{\vp}([x_j]))_{j=1}^m$ is a packing in $P^{n-1}$ is exactly $\vp=\max_{j\neq i}\frac{1}{2}\cos^{-1}(|\langle x_i,x_j\rangle|)$.    This gives the following corollary which solves Problem \ref{Q:2} for the case $m=n+1$, and solves Problem \ref{Q:1} for $2^{-1/2}(1+1/n)^{1/2}\leq\lambda<1$.

\begin{cor}\label{C:MB}
    Let $(x_j)_{j=1}^{n+1}\subseteq S_{\R^n}$ and let $\alpha=\max_{i\neq j}|\langle x_i,x_j\rangle|$.  If $(x_j)_{j=1}^{n+1}$ is full spark then the critical value for   $(x_j)_{j=1}^{n+1}$ to do $\lambda_c$-saturation recovery on $B_{\R^n}$ is given by $\lambda_c=2^{-1/2}(1+\alpha)^{1/2}$. 
    Furthermore, $\lambda_c\geq 2^{-1/2}(1+1/n)^{1/2}$ and equality is achieved if and only if  $(x_j)_{j=1}^{n+1}$ is an equiangular frame.

    If $(x_j)_{j=1}^{n+1}$ is not full spark then $\lambda_c=1$ is the critical value for   $(x_j)_{j=1}^{n+1}$ to do $\lambda_c$-saturation recovery on $B_{\R^n}$.
\end{cor}
\begin{proof}
We first assume that $(x_j)_{j=1}^{n+1}$ is full spark.
 The greatest $\vp>0$ so that $(B^o_\vp[x_j])_{j=1}^{n+1}$ is a packing in $P^{n-1}$ is given by $\vp=\frac{1}{2}\cos^{-1}(\alpha)$.  By Proposition \ref{P:pack}, the critical value $\lambda_c$ for $(x_j)_{j=1}^{n+1}$ to do $\lambda_c$-saturation recovery on $B_{\R^n}$ is given by $\lambda_c=\cos(\frac{1}{2}\cos^{-1}(\alpha))=2^{-1/2}(1+\alpha)^{1/2}$.  
  
  The Welch bound \eqref{E:Welch} gives that $\alpha\geq n^{-1}$ and equality is achieved if and only if $(x_j)_{j=1}^{n+1}$ is an equiangular tight frame by \cite{SH03}.  

  We now assume that  $(x_j)_{j=1}^{n+1}$ is not full spark.  Thus, there exists $1\leq k\leq n+1$ such that $(x_j)_{j\neq k}$ is not a frame of $\R^n$.  As $(x_j)_{j\in J_{1}^{\sharp}(x_{k})}=(x_j)_{j\neq k}$ we have by Theorem \ref{T:no} that $\lambda_c=1$. 
\end{proof}

Given $(x_j)_{j=1}^m\subseteq S_{\R^n}$, it is simple to determine the greatest radius so that a collection of balls 
centered at $([x_j])_{j=1}^m$ is a packing in $P^{n-1}$, but when $L\geq 2$ it becomes much more difficult to determine the greatest radius so that a collection of balls 
centered at $([x_j])_{j=1}^m$ is a $L$-fold multi-packing in $P^{n-1}$.  Because of this, the following corollary gives bounds for the solutions of  Problem  \ref{Q:1} and Problem  \ref{Q:2}, but does not provide complete solutions.

\begin{cor}
    Let $(x_j)_{j=1}^m\subseteq S_{\R^n}$ be a full spark frame of $\R^n$. Let 
    \begin{equation}\label{E:a}
    \alpha=\max_{1 \leq k_0 < \cdots < k_{m-n} \leq m}\,\,\min_{0 \leq i<j \leq m-n} |\langle x_{k_i}, x_{k_j}  \rangle |.        
    \end{equation}

    Then the critical value for   $(x_j)_{j=1}^m$ to do $\lambda_c$-saturation recovery on $B_{\R^n}$ satisfies
$$\alpha\leq \lambda_c \leq 2^{-1/2}(1+\alpha)^{1/2}.$$
\end{cor}

\begin{proof}
Suppose that $(k_j)_{j=0}^{m-n}$ satisfies \eqref{E:a}.  Then $x_{k_j}\not\in J_{\alpha}^{\sharp}(x_{k_0})$ for all $0\leq j\leq m-n$.  Thus, $|J_{\alpha}^{\sharp}(x_{k_0})|<n$ and hence $(x_j)_{j\in J_{\alpha}^{\sharp}(x_{k_0})}$ cannot be a frame for $\R^n$.  This gives that $\alpha\leq \lambda_c$.

We now choose $x\in\R^n$ so that $(x_j)_{j\in J_{\lambda_c}^{\sharp}(x)}$ is not a frame of $\R^n$.  As $(x_j)_{j\in J}$ is full spark, it must be that  $|J_{\lambda_c}^{\sharp}(x)|<n$. Thus, there exists $(k_j)_{j=0}^{m-n}$ so that $|\langle x,x_j\rangle|\geq \lambda_c$ for all $0\leq j\leq m-n$.  It follows that $|\langle x_{k_i}, x_{k_j}  \rangle |\geq 2\lambda_c^2-1$ for all $0\leq i< j\leq m-n$.  
Thus, $\alpha\geq 2\lambda_c^2-1$ and 
$\lambda_c\leq 2^{-1/2}(1+\alpha)^{1/2}$.
\end{proof}

\section{Stability of saturation recovery}\label{S:Stab}

Recall that a frame $(x_j)_{j\in J}$ for a Hilbert space $H$ is said to do $\lambda$-saturation recovery on a set $X\subseteq\R^n$ if every $x\in X$ may be recovered from $\Phi_\lambda \Theta=(\phi_\lambda(\langle x,x_j\rangle)_{j\in J}$.  In applications, it is important that we not only be able to recover $x$, but also that this recovery be stable.  We say that $(x_j)_{j\in J}$ does $C$-stable $\lambda$-saturation recovery on $X$ if the recovery of $x$ from $\Phi_\lambda \Theta=(\phi_\lambda(\langle x,x_j\rangle)_{j\in J}$ is $C$-Lipschitz.  That is,  a frame $(x_j)_{j\in J}$ of $H$ with analysis operator $\Theta:H\rightarrow \ell_2(J)$ does {\em  $C$-stable $\lambda$-saturation recovery on $X$} if for all $x,y\in X$ we have that

\begin{equation}
\|x-y\|_{H}\leq C\big\|\Phi_\lambda \Theta x-\Phi_\lambda \Theta y\big\|_{\ell_2(J)}=C\Big(\sum_{j\in J}|\phi_\lambda(\langle x,x_j\rangle)-\phi_\lambda(\langle x,x_j\rangle)|^2  \Big)^{1/2}.
\end{equation}

It is well known that a frame for a finite dimensional Hilbert space does phase retrieval if and only if it does stable phase retrieval. This fundamental
result in the mathematics of phase retrieval has been proven in multiple papers using a variety
of methods \cite{AAFG24, AG17, BCMN14, BW15, CCD16, FOPT23}.
We now prove that this is also the case for $\lambda$-saturation recovery on $B_{\R^n}$ when $\lambda$ is strictly greater than the critical saturation level.

\begin{thm}
Let $(x_j)_{j\in J}$ be a frame for $\R^n$, and let $\lambda_c$ be the critical value for   $(x_j)_{j\in J}$ to do $\lambda_c$-saturation recovery on $B_{\R^n}$. 
  Then for all $\lambda>\lambda_c$
there exists $C_\lambda>0$ so that $(x_j)_{j\in J}$ does $C_\lambda$-stable $\lambda$-saturation recovery.
\end{thm}
\begin{proof}
 Let $A>0$ be such that $A$ is a lower frame bound of $(x_j)_{j\in I}$ for every set $I\subseteq J$ such that $(x_j)_{j\in I}$ is a frame of $H$.  Let $B$ be an upper frame bound of $(x_j)_{j\in J}$.

Let $\lambda>\lambda_c$ and let $1>\alpha>0$ so that $\alpha\lambda>\lambda_c$.  
 We first consider $x,y\in B_H$ such  that $\|x-y\|\leq B^{-1} (1-\alpha)\lambda$.  Let $j\in J_{\alpha\lambda}(x)$ we have that 
\begin{align*}
|\langle y,x_j\rangle|&\leq |\langle x,x_j\rangle|+|\langle x-y,x_j\rangle|\\
&\leq \alpha\lambda+\|x-y\| \|x_j\|\\
&\leq \alpha\lambda+(B^{-1} (1-\alpha)\lambda)B=\lambda.
\end{align*}
Thus, if $\|x-y\|\leq B^{-1} (1-\alpha)\lambda$ then $J_{\lambda}(y)\supseteq  J_{\alpha\lambda}(x)$.  Hence we have that
\begin{align*}
A\|x-y\|^2&\leq \sum_{j\in J_{\alpha \lambda}(x)} |\langle x-y,x_j\rangle|^2\\
&= \sum_{j\in J_{\alpha \lambda}(x)} |\phi_\lambda(\langle x,x_j\rangle)-\phi_\lambda(\langle y,x_j\rangle)|^2\\
&\leq \sum_{j\in J} |\phi_\lambda(\langle x,x_j\rangle)-\phi_\lambda(\langle y,x_j\rangle)|^2.
\end{align*}
 We now consider the compact set $X:=\{(x,y)\in B_H\times B_H : \|x-y\|\geq B^{-1} (1-\alpha)\lambda\}$.  As $(x_j)_{j\in J}$ does $\lambda$-saturation recovery we have that $\sum_{j\in J} |\phi_\lambda(\langle x,x_j\rangle)-\phi_\lambda(\langle y,x_j\rangle)|^2>0$ for all $(x,y)\in X$.  As $X$ is compact, there exists $\delta>0$ so that $\sum_{j\in J} |\phi_\lambda(\langle x,x_j\rangle)-\phi_\lambda(\langle y,x_j\rangle)|^2\geq \delta$ for all $(x,y)\in X$.  We have that $\|x-y\|^2\leq 2$ for all $(x,y)\in X$ and hence
$$
2^{-1}\delta\|x-y\|^2\leq \sum_{j\in J} |\phi_\lambda(\langle x,x_j\rangle)-\phi_\lambda(\langle y,x_j\rangle)|^2\hspace{1cm}\textrm{ for all }(x,y)\in X.$$
Thus,  $(x_j)_{j\in J}$ does $C_\lambda$-stable $\lambda$-saturation recovery for $C_\lambda=\max(A^{-1/2},2^{1/2}\delta^{-1/2})$.
\end{proof}

This leaves the following open question.

\begin{prob}\label{Prob:C}
Let $(x_j)_{j=1}^m$ be a frame for $\R^n$, and let $\lambda_c$ be the critical value for   $(x_j)_{j=1}^m$ to do $\lambda_c$-saturation recovery on $B_{\R^n}$. 
  Does $(x_j)_{j\in J}$ do stable  $\lambda_c$-saturation recovery on $B_{\R^n}$?
\end{prob}

We say that $(y_j)_{j\in J}$ is an $\vp$-perturbation of a frame $(x_j)_{j\in J}$ if $\sum_{j\in J}\|x_j-y_j\|^2<\vp$.
If $(x_j)_{j\in J}$ is a frame of a Hilbert space $H$ then there exists $\vp>0$ so that  every $\vp$-perturbation of  $(x_j)_{j\in J}$  is a frame of $H$  \cite{C95}.
Likewise, if $(x_j)_{j\in J}$ is a frame for a finite dimensional Hilbert space which does  phase retrieval then there exists $\vp>0$ so that every $\vp$-perturbation of  $(x_j)_{j\in J}$  does phase retrieval \cite{B17,AFGLS23}.  It follows from Theorem \ref{T:no} that if $(x_j)_{j\in J}$ is a frame of $\R^n$ and $\lambda>\lambda_c$ then there exists $\vp>0$ so that every $\vp$-perturbation of  $(x_j)_{j\in J}$  does $\lambda$-saturation recovery on $B_{\R^n}$.  However, for each $\vp>0$ it is simple to construct an $\vp$-perturbation of $(x_j)_{j\in J}$ which does not do $\lambda_c$-saturation recovery on $B_{\R^n}$.  Indeed, if $\vp_j>0$ for all $j\in J$ and $\sum_{j\in J}\vp_j^2<\vp$ then 
choosing $(y_j)_{j\in J}=((1+\vp_j)x_j)_{j\in J}$ will suffice.
This makes solving Problem \ref{Prob:C} significantly more difficult as stability of $\lambda_c$-saturation recovery on $B_{\R^n}$ is not preserved under small perturbations.

We now consider the problem of doing stable $\lambda$-saturation recovery on sets other than $B_{\R^n}$.

\begin{thm}
    Let $(x_j)_{j\in J}$ be a finite frame for $\R^n$.  For each $x\in\R^n$, $\lambda>0$, and $C>0$ the following are equivalent.  
 \begin{enumerate}[label=\textup{(\arabic*)}]
     \item $(x_j)_{j\in J}$ does $C$-stable $\lambda$-saturation recovery on an open set containing $x$. 
     \item $(x_j)_{j\in J_{\lambda}^{\sharp}(x)}$ is a frame of $\R^n$ with lower frame bound $C^{-2}$.
        \end{enumerate}
\end{thm}

\begin{proof}
    We first assume that $(x_j)_{j\in J_{\lambda}^{\sharp}(x)}$ is a frame of $\R^n$ with lower frame bound $C^{-2}$.  As $J$ is a finite set, we may choose a neighborhood $U$ of $x$ so that $|\langle y,x_j\rangle|<\lambda$ for all $y\in U$ and $j\in J_{\lambda}^{\sharp}(x)$.  For all $y,z\in U$ we have that
    $$C^{-2}\|y-z\|^2\leq \sum_{j\in  J_{\lambda}^{\sharp}(x)}|\langle y-z,x_j\rangle|^2=\sum_{j\in  J_{\lambda}^{\sharp}(x)}|\phi_\lambda(\langle y,x_j\rangle)-\phi_\lambda(\langle z,x_j\rangle)|^2.
    $$
Thus, $(x_j)_{j\in J}$ does $C$-stable $\lambda$-saturation recovery on $U$.

We now assume that $(x_j)_{j\in J}$ does $C$-stable $\lambda$-saturation recovery on some open neighborhood $U$ of $x$.  Let $R>1$ be sufficiently close to $1$ such that $Rx\in U$, $J_{\lambda}^{\sharp}(Rx)=J_{\lambda}^{\sharp}(x)$, and that $|\langle Rx,x_j\rangle|\neq \lambda$ for all $j\in J$.  Let $y\in\R^n$ and choose $a>0$ small enough so that $Rx+ay\in U$, $J_{\lambda}^{\sharp}(Rx+ay)=J_{\lambda}^{\sharp}(Rx)$, and $\mathrm{sign}(\langle Rx+ay,x_j\rangle)=\mathrm{sign}(\langle Rx,x_j\rangle)$ for all $j\in (J_{\lambda}^{\sharp}(Rx))^{c}$.  We have that
\begin{align*}
    \|ay\|^2&=\|Rx+ay-Rx\|^2\\
    &\leq C^2\sum_{j\in J}|\phi_\lambda(\langle Rx+ay,x_j\rangle)-\phi_\lambda(\langle Rx,x_j\rangle)|^2\hspace{1cm}\textrm{ as }Rx,(Rx+ay)\in U,\\
    &= C^2\sum_{j\in J_{\lambda}^{\sharp}(x)}|\langle Rx+ay,x_j\rangle-\langle Rx,x_j\rangle|^2\\
    &= C^2\sum_{j\in J_{\lambda}^{\sharp}(x)}|\langle ay,x_j\rangle|^2.
\end{align*}
Thus, we have that $C^{-2}\|y\|\leq \sum_{j\in J_{\lambda}^{\sharp}(x)}|\langle ay,x_j\rangle|^2$.  Hence, $C^{-2}$ is a lower frame bound for $(x_j)_{j\in J_{\lambda}^{\sharp}(x)}$.
\end{proof}

\section{The frame algorithm }\label{S:FA}

If $(x_j)_{j\in J}$ is a frame of a Hilbert space $H$ with analysis operator $\Theta:H\rightarrow \ell_2(J)$ then the {\em frame operator} $\Theta^*\Theta:H\rightarrow H$ is given by $\Theta^*\Theta x=\sum \langle x,x_j\rangle x_j$.  As $(x_j)_{j\in J}$ is a frame, we have that the analysis operator $\Theta:H\rightarrow \ell_2(J)$ is an isomorphic embedding, and the frame operator 
$\Theta^*\Theta:H\rightarrow H$ is bounded and has bounded inverse.  Furthermore, if $(x_j)_{j\in J}$  has frame bounds $0<A\leq B$ then the frame operator satisfies $\|\Theta^*\Theta\|\leq B$ and $\|(\Theta^*\Theta)^{-1}\|\leq A^{-1}$.  We may thus recover any vector $x\in H$ from the frame coefficients $\Theta x=(\langle x,x_j\rangle)_{j\in J}$ by applying the linear operator $(\Theta^*\Theta)^{-1}\Theta^*$.  Doing this in applications would require inverting a matrix which can be computationally expensive.  However, there exist computationally efficient algorithms which use the frame coefficients $\Theta x=(\langle x,x_j\rangle)_{j\in J}$ to build a sequence of vectors $(y_n)_{n=0}^\infty\subseteq H$ that converges to $x$ exponentially fast.  The {\em frame algorithm} uses a parameter $0<\alpha<2/B$ to build a sequence $(y_n)_{n=0}^\infty\subseteq H$, as explained in the following theorem.  The frame algorithm is a well known tool (see for example \cite[Theorem III]{DS52} or \cite{CKF13}) and the basic framework was shown by Gr\"ochenig to allow for even faster convergence with certain modifications \cite{G93}.

\begin{thm}\label{T:FA}(Frame algorithm)
 Let $(x_j)_{j\in J} $ be a frame of a  Hilbert space $ H $ with frame bounds $ A, B$ and analysis operator $\Theta:H\rightarrow\ell_2(J) $.  Let $0<\alpha<B/2$.  Given an element in the range of the analysis operator $ \Theta x\in \ell_2(J)$,  define a sequence $(y_k)^\infty_{k=0} $ in $ H$ by $y_0=0$ and 
 \begin{equation}\label{E:FA}
 y_{k+1}:=y_{k} + \alpha\, \Theta^* (\Theta x-\Theta y_{k})= y_{k} + \alpha\, \sum_{j\in J} \langle x-y_k,x_j\rangle x_j \hspace{.5cm}\textrm{ for all } k \geq 0.
 \end{equation}
 Then   $\Vert x- y_{k+1}\Vert \le C_{\alpha} \Vert x-y_k \Vert$ for all $k\geq 0$, where $C_{\alpha}: = \max{\lbrace \vert 1-\alpha A\vert, \vert 1-\alpha B\vert \rbrace}$. Thus,  $(y_k)^\infty_{k=0} $ converges to $x$ and satisfies
 $$ \|x-y_{k}\|\leq C_\alpha^k \|x\| \hspace{1cm}\textrm{ for all } k \geq 0. $$
 Note that $\frac{B-A}{B+A}$ is the optimal value for $C_{\alpha}$ and it occurs when $\alpha = \frac{2}{A+B}$. 
 \end{thm}

Our goal is to adapt the frame algorithm to the non-linear problem of recovering a vector from saturated measurements.  For the sake of clarity, we will refer to the algorithm presented in Theorem \ref{T:FA} as the {\em linear frame algorithm}.  Let $(x_j)_{j\in J}$ be a frame of a Hilbert space $H$ and let  $\lambda>0$.  
 For $x\in H$ we denote the following sets,
\begin{align*}
&\textrm{Unsaturated coordinates:}\hspace{.5cm} &&J_{\lambda}(x)=\{j\in J:|\langle x,x_j\rangle|\leq \lambda\},\\
&\textrm{Positively saturated coordinates:} && J_{\lambda}^{+}(x)=\{j\in J:\langle x,x_j\rangle > \lambda\},\\
&\textrm{Negatively saturated coordinates:} && J_{\lambda}^{-}(x)=\{j\in J:\langle x,x_j\rangle < -\lambda\}.
\end{align*}
It will also be convenient to define specific subsets of $J_{\lambda}^{+}(x)$ and $J_{\lambda}^{-}(x)$, respectively, relative to a fixed element $y\in H$.  Thus, we also define
\begin{align*}
J_{\lambda}^{+}(x,y)&=\{j\in J_{\lambda}^{+}(x):\langle y,x_{j}\rangle<\lambda\}, \\
J_{\lambda}^{-}(x,y)&=\{j\in J_{\lambda}^{-}(x):\langle y,x_{j}\rangle>-\lambda\}.
\end{align*}
We now recursively define the {\em $\lambda$-saturated frame algorithm} for recovering a vector $x\in H$ from the values $(\phi_\lambda(\langle x,x_j\rangle))_{j\in J}$.  We set $y_0=0$ and for $k\in\N_0$ and $y_k\in H$ we choose $\alpha_k,\beta_k\geq 0$ and let
\begin{equation}\label{E:SFA}
\begin{split}
y_{k+1}=y_k+\alpha_k&\sum_{j\in J_{\lambda}(x)}\Big(\langle x,x_j\rangle-\langle y_k,x_j\rangle\Big)x_j
+\beta_k\sum_{j\in J_{\lambda}^{+}(x,y_k)}\Big(\lambda- \langle y_k,x_j\rangle\Big)x_j\\
&+\beta_k\sum_{j\in J_{\lambda}^{-}(x,y_k)}\Big(-\lambda-\langle y_k,x_j\rangle\Big)x_j.
\end{split}
\end{equation}
Intuitively, the terms where $j\in J_{\lambda}^{+}(x,y_{k})$ are working to make $\phi_\lambda(\langle y_{k+1},x_j\rangle)$ closer to $\phi_\lambda(\langle x,x_j\rangle)=\lambda$, and the terms with $j\in J_{\lambda}^{-}(x,y)$ are working to make $\phi_\lambda(\langle y_{k+1},x_j\rangle)$ closer to $\phi_\lambda(\langle x,x_j\rangle)=-\lambda$.  Note that we know the vector $y_k\in H$ and thus we are able to use the actual frame coefficients $(\langle y_k,x_j\rangle)_{j\in J}$ rather than the saturated frame coefficients $(\phi_\lambda(\langle y_k,x_j\rangle))_{j\in J}$.
In the linear frame algorithm, Theorem \ref{T:FA_P} gives that the optimal choice for the scalar $\alpha$ is given by $\alpha=2/(A+B)$ where $A$ and $B$ are the frame bounds for $(x_j)_{j\in J}$.  However, the optimal choice for $\alpha_k$ and $\beta_k$ in the $\lambda$-saturated frame algorithm can change each step.

Algorithms for doing $\lambda$-saturation recovery typically either use the saturated terms to give a bound that the vector should satisfy, or throw out the saturated terms and reconstruct using only the unsaturated terms.  
Our algorithm can be adapted to either method by choosing different values for the parameters $\alpha_k,\beta_k\geq0$.  If we choose $\beta_k=0$ and let $\alpha_k=2/(A_x+B_x)$, where $A_x$ and $B_x$ are the frame bounds for $(x_j)_{j\in J_{\lambda}(x)}$, then applying the $\lambda$-saturated frame algorithm to $(\phi_\lambda(\langle x,x_j\rangle))_{j\in J}$ is the same as applying the optimal linear frame algorithm to $(\langle x,x_j\rangle)_{j\in J_{\lambda}(x)}$.  There is not one best approach to doing saturation recovery, in general, and we consider how the $\lambda$-saturated frame algorithm performs in different situations.

When working with a particular frame $(x_j)_{j\in J}$ it is expected that the frame bounds of $(x_j)_{j\in J}$ are known, but it is not reasonable to assume that frame bounds have been calculated for every subset of $(x_j)_{j\in J}$. Thus,   
when doing saturation recovery using only the unsaturated frame coefficients, one can choose the parameter $\alpha$ in \eqref{E:FA} based on the full frame $(x_j)_{j\in J}$ or one can calculate the frame bounds for $(x_j)_{j\in J_{\lambda}(x)}$ and use that to choose the corresponding optimal parameter $\alpha$.  In the following theorem, we show that (except in trivial circumstances where they are equivalent) if one chooses the parameter $\alpha$ based on the full frame $(x_j)_{j\in J}$ then the $\lambda$-saturated frame algorithm always outperforms the linear frame algorithm based solely on the unsaturated coordinates.

\begin{thm}\label{T:satfa}
    Let $(x_j)_{j\in J}$  be a frame for a Hilbert space $H$ with frame bounds $A\leq B$. 
  Let $x\in H$ and suppose that $(y_k)_{k=0}^\infty\subseteq H$ is constructed by the $\lambda$-saturated frame algorithm with $\alpha_k=\beta_k=2/(A+B)$ for all $k\in\N_0$.    For each $k\in\N_0$, if the optimal lower frame bound for $(x_j)_{j\in J_{\lambda}(x)}$ is strictly less than the optimal lower frame bound for $(x_j)_{j\in J_{\lambda}(x)\cup J_{\lambda}^{+}(x,y_k)\cup J_{\lambda}^{-}(x,y_{k})}$ then there exists $\vp_{x,y_k}>0$ so that
\begin{equation}
    \|x-y_{k+1}\|\leq(1-\vp_{x,y_k})C_{2/(A+B)}\|x-y_k\|,
\end{equation}
  where  $C_{2/(A+B)}$ is the constant in Theorem \ref{T:FA} for applying the linear frame algorithm to $(x_j)_{j\in J_{\lambda}(x)}$ with coefficient $\alpha=2/(A+B)$.
\end{thm}

\begin{proof}
 Let $k\in\N_0$ and denote  $J_{S}=J_{\lambda}^{+}(x,y_k)\cup J_{\lambda}^{-}(x,y_{k})$.  
For each $j\in J_{S}$ we let $\gamma_j=(\phi_\lambda(\langle x,x_j\rangle)-\langle y_k,x_j\rangle)/(\langle x-y_k,x_j\rangle)$. Note that $1>\gamma>0$ for all $j\in J_S$.
 Let $A',B'> 0$ be the optimal frame bounds for  $(x_j)_{j\in J_{\lambda}(x)}$, and let $A'',B''>0$ be the optimal frame bounds for the frame $( x_j)_{j\in J_{\lambda}(x)}\cup(\gamma^{1/2} x_j)_{j\in J_{S}}$.  
 Suppose that the optimal lower frame bound for 
$(x_j)_{j\in J_{\lambda}(x)\cup J_{S}}$ is strictly greater than $A'$.  As $0<\gamma_j<1$ for all $j\in J_{S}$ we have that $A'<A''\leq A$ and $B'\leq B''\leq B$.
By Theorem \ref{T:FA}, applying the linear frame algorithm to $(x_j)_{j\in J_{\lambda}(x)}$ using $x,y_k\in H$ with $\alpha=2/(A+B)$ gives the constant $C_{2/(A+B)}=1-2A'/(A+B)$ as $\alpha<2/(A'+B')$.

 Note that applying the $\lambda$-saturated frame algorithm with the frame $(x_j)_{j\in J}$ using $x,y_k\in H$ with $\alpha_k=\beta_k=2/(A+B)$ will produce the same vector $y_{k+1}$ as applying the linear frame algorithm with the frame  $(x_j)_{j\in J_{\lambda}(x)}\cup(\gamma^{1/2} x_j)_{j\in J_{S}}$ using $x,y_k\in H$ with $\alpha=2/(A+B)$. 
Thus,
\begin{align*}
\|x-y_{k+1}\|&\leq (1-2A''/(A+B))\|x-y_k\|\hspace{2cm}\textrm{ by Theorem \ref{T:FA}},\\
&\leq(1-\vp_{x,y_k})(1-2A'/(A+B))\|x-y_k\|\hspace{1cm}\textrm{ for some }0<\vp_{x,y_k}<1,\\
&=(1-\vp_{x,y_k})C_{2/(A+B)}\|x-y_k\|.
\end{align*}
\end{proof}

We now consider the  $\lambda$-saturated frame algorithm for the case that $(x_j)_{j\in J}$ is a Parseval frame and take $\alpha_k=\beta_k=1$ for all $k\in \N$. Note in the case of a Parseval frame, if none of the coordinates are saturated then $x=\sum_{j\in J}\langle x,x_j\rangle x_j$ and the frame algorithm will terminate at $x$ immediately.  If there are saturated coefficients however then the frame algorithm may not reach $x$ in finitely many steps, but we prove that it converges exponentially fast.  The convergence rate of the linear frame algorithm depends on the frame bounds, and we will see that the convergence rate of the $\lambda$-saturated frame algorithm may be estimated by the stability constant of $\lambda$-saturation recovery on the ball $2\|x\|B_H$. We make this assumption about stability on a larger ball than $\|x\| B_H$ as when building the sequence $(y_n)_{n=0}^\infty$, it is possible for the $\lambda$-saturated frame algorithm to go outside the ball of radius $\|x\|$.  However, we expect that a much smaller radius than $2\|x\|$ would be sufficient.
\begin{thm}\label{T:FA_P}
Let $(x_j)_{j\in J}$ be a Parseval frame for a Hilbert space $H$.  Let $x\in H$ and suppose that $(x_j)_{j\in J}$  does $C$-stable $\lambda$-saturation recovery on $2\|x\| B_H$.  If $(y_k)_{k=0}^\infty\subseteq H$ is constructed by the $\lambda$-saturated frame algorithm with $\alpha_k=\beta_k=1$ for all $k\geq0$ then 
  $\|x-y_{k+1}\|^2 \leq ( 1-C^{-2} )\|x-y_{k}\|^2$ for all $k\geq 0$.  Hence,
 $$\|x-y_{k}\|\leq ( 1-C^{-2} )^{k/2}\|x\|\hspace{1cm}\textrm{ for all }k\geq0.$$
\end{thm}
\begin{proof}
As $y_0=0$, we trivially have that $\|x-y_0\|\leq \|x\|$.
We now fix $k\in\N_0$ and assume that $\|x-y_k\|\leq ( 1-C^{-2} )^{k/2}\|x\|$.   Hence $y_k\in 2\|x\|B_H$. We denote the coordinates that get used in the $\lambda$-saturated frame algorithm as $J_F=J_{\lambda}(x)\cup J_{\lambda}^{+}(x,y_k)\cup J_{\lambda}^{-}(x,y_{k})$.

We have \begin{equation}\label{E:Pnp1}
x-y_{k+1}=x-y_k-\sum_{j\in J_F}\Big(\phi_\lambda(\langle x,x_j\rangle)-\langle y_k,x_j\rangle\Big)x_j.
\end{equation}
 The fact that $(x_j)_{j\in J}$ is a Parseval frame implies that 
\begin{equation}\label{E:Pn}
   x- y_{k}= \sum_{j\in J} \langle x- y_{k},x_j \rangle x_j. 
\end{equation}
Thus, we may substitute \eqref{E:Pn} into \eqref{E:Pnp1} to obtain 
 \begin{equation}
  x- y_{k+1}=\sum_{j\in J^c_F} \langle x- y_{k},x_j \rangle x_j+ \sum_{j\in J_F} \bigg[\langle x-y_{k}, x_j\rangle- \Big(\phi_{\lambda} (\langle x,x_j \rangle)-  ( \phi_{\lambda} (\langle y_{k},x_j \rangle)\Big)\bigg]x_j.   
 \end{equation}
We consider the vector $f=\big(\langle x-y_{k}, x_j\rangle\big)_{j\in J_F^c}\cup\big(\langle x-y_{k}, x_j\rangle - (\phi_{\lambda} (\langle x,x_j \rangle)-  ( \phi_{\lambda} (\langle y{k},x_j \rangle\big)_{j\in J_F}\in\ell_2(J)$.  Let $\Theta:H\rightarrow \ell_2(J)$ be the analysis operator of the frame $(x_j)_{j\in J}$ and let $\Theta^*:\ell_2(J)\rightarrow H$ be the synthesis operator.  Note that $x-y_{n+1}=\Theta^* f$ and that $\|\Theta\|=\|\Theta^*\|=1$ as $(x_j)_{j\in J}$ is a Parseval frame.  We now have that, 
  \begin{align*}
  \|x-y_{k+1}\|^2&=\|\Theta^{*} f\|^2\\
  &\leq \|f\|^2\hspace{.5cm}\textrm{ as }\|\Theta^*\|=1,\\
 &=\sum_{j\in J_F^c}\bigg | \langle x-y_k, x_j\rangle  \bigg |^2+\sum_{j\in J_F}\bigg | \langle x-y_k, x_j\rangle - \Big(\phi_{\lambda} (\langle x,x_j \rangle)-   \langle y_k,x_j \rangle\Big) \bigg |^2\\
 &=\sum_{j\in J_F^c}\bigg | \langle x-y_k, x_j\rangle  \bigg |^2+\sum_{j\in J_F} \bigg | \big|\langle x-y_k, x_j\rangle\big| - \big|\phi_{\lambda} (\langle x,x_j \rangle)-  \langle y_k,x_j \rangle \big|\bigg |^2\\
 &\hspace{2cm}\textrm{ as }  \mathrm{sign}(\langle x-y_k, x_j\rangle)=\mathrm{sign}(  \phi_{\lambda} (\langle x,x_j \rangle)-   \langle y_k,x_j \rangle)\textrm{ for }j\in J_F,  \\
  &\leq \sum_{j\in J}\bigg | \langle x-y_n, x_j\rangle  \bigg |^2 - \sum_{j\in J_F}\bigg | \phi_{\lambda} (\langle x,x_j \rangle)-  \langle y_k,x_j \rangle \bigg |^2 \\
    &\leq \sum_{j\in J}\bigg | \langle x-y_n, x_j\rangle  \bigg |^2 - \sum_{j\in J_F}\bigg | \phi_{\lambda} (\langle x,x_j \rangle)-  \phi_{\lambda} (\langle y_k,x_j \rangle )\bigg |^2 \\
  &=\|x-y_n\|^2-\sum_{j\in J} \bigg | \phi_{\lambda} (\langle x,x_j \rangle)-   \phi_{\lambda} (\langle y_k,x_j \rangle) \bigg |^2 \\
  &\leq \|x-y_k\|^2-C^{-2}\|x-y_k\|^2\textrm{ as $(x_j)_{j\in J}$ does $C$-stable $\lambda$-saturation recovery on $2\|x\|B_H$.}
 \end{align*}

Thus, we have that $\|x-y_{k+1}\|^2\leq (1-C^{-2})\|x-y_k\|^2$. 
\end{proof}

In Theorem \ref{T:FA_P} we assumed that the frame $(x_j)_{j\in J}$ does $C$-stable saturation recovery on $2\|x\|B_H$.  We make this assumption because the $\lambda$-saturated frame algorithm can build a sequence $(y_n)_{n=0}^\infty$ where $\|y_n\|>\|x\|$ for some $n\in\N$.  We always have that $\|y_1\|\leq \|x\|$ when $(x_j)_{j\in J}$ is a Parseval frame, but in the following proposition we give a simple condition which guarantees that $\|y_2\|>\|x\|$.
\begin{prop}\label{P:badFA}
Let $(x_j)_{j=1}^N$ be a Parseval frame of a finite dimensional Hilbert space $H$ and suppose that $x\in H$ and $\lambda>0$ are such that 
\begin{enumerate}[label=\textup{(\alph*)}]
\item $\langle x,x_1\rangle >\lambda$,
\item $\langle x,x_2\rangle =\lambda$,
\item $|\langle x,x_j\rangle| + |\langle x,x_1\rangle-\lambda||\langle x_1,x_j\rangle|\leq\lambda$\hspace{1cm} for $2<j\leq N$,
\item $\langle x,x_1\rangle \|x_1\|^2 < -\langle x_1,x_2\rangle \lambda$.
\end{enumerate}
Then,  if $(y_n)_{n=0}^\infty$ is created by the $\lambda$-saturated frame algorithm we have that  $\|y_2\|>\|x\|$.
\end{prop}

\begin{proof}
As $|\langle x,x_j\rangle|\leq \lambda$ for all $1< j\leq N$, $\langle x,x_1\rangle>\lambda$, and $(x_j)_{j=1}^N$ is a Parseval frame, we have that
\begin{equation}\label{E:y0}
y_1=x -(\langle x,x_1\rangle-\lambda)x_1.
\end{equation}
Condition (d) implies that 
 $\|x_1\|< 1$ and hence it follows from \eqref{E:y0} that  $\langle y_1,x_1\rangle> \lambda$.  As $\langle x,x_2\rangle =\lambda$ and $\langle x_2,x_1\rangle <0$ we have that $\langle y_1,x_2\rangle >\lambda$.  By (c) we have that $|\langle y_1,x_j\rangle|\leq \lambda$ for all $2<j\leq N$.
Using that $(x_j)_{j=1}^N$ is a Parseval frame, we may simplify the formula for $y_2$ to obtain
\begin{equation}\label{E:y1}
y_2=x -(\langle x,x_1\rangle-\lambda)\|x_1\|^2 x_1-(\langle x,x_1\rangle-\lambda)\langle x_1,x_2\rangle x_2.
\end{equation}
By expanding $\|y_2\|^2=\langle y_2,y_2\rangle$ and applying (d), one can check that $\|y_2\|>\|x\|$.

\end{proof}

\section{Numerical Implementation}\label{S:exp}

The $\lambda$-saturated frame algorithm described by \eqref{E:SFA} allows for substantial flexibility in implementation due to the fact that the constants $\alpha_{k}$ and $\beta_{k}$ can be chosen at each iteration.  Recall that the coefficient $\alpha_{k}$ is used to control the contribution of the unsaturated frame coefficients of the vector $x$, while $\beta_{k}$ limits the contribution of any saturated frame coefficients included for that iteration.  If $\beta_{k}=0$, then the $\lambda$-saturated frame algorithm uses only the unsaturated frame coefficients of $x$ and coincides with the linear frame algorithm corresponding to the unsaturated frame vectors.  In order to demonstrate the benefit of including saturated frame coefficients in the $\lambda$-saturated frame algorithm, as described by Theorem \ref{T:satfa}, we will present the results of a series of numerical experiments.  Our results are in line  with previous experiments which have shown that algorithms which use all the coordinates typically outperform algorithms which use only the unsaturated coordinates \cite{LBDB11}.

The numerical experiments, implemented in MATLAB, will compare two distinct implementations of the $\lambda$-saturated frame algorithm using randomly generated frames consisting of 30 unit vectors in $\mathbb{R}^{10}$ for the recovery of a randomly chosen unit vector $x\in \mathbb{R}^{10}$.  The first implementation corresponds to the linear frame algorithm applied to the unsaturated frame vectors, using $\alpha_{k} = 2/(A+B)$ and $\beta_{k}=0$.  Here, $A$ and $B$ are the optimal frame bounds for the entire frame $(x_{j})_{j=1}^{30}$.  The second implementation is nonlinear and uses $\alpha_{k} = \beta_{k} = 2/(A+B)$.  Each experiment uses a fixed positive value for $\lambda$ and consists of 1000 trials.  The frame $(x_{j})_{j=1}^{30}$ and vector $x$ are chosen randomly for each trial and the error $\Vert y_{k}-x\Vert$ is recorded for each iteration $1\le k\le 50$.  Unit vectors are constructed by first generating a vector in $\mathbb{R}^{10}$ with coordinates chosen according to the standard normal distribution and then rescaling the resulting vector to have unit norm.

Figures \ref{fig:error} and \ref{fig:error_red} show the results of four experiments corresponding to $\lambda = 0.16$, $\lambda=0.24$, $\lambda=0.32$, and $\lambda=0.4$.  In Figure \ref{fig:error}, the data markers represent the mean error $\Vert y_{k}-x\Vert$ for each algorithm as a function of the iteration number, plotted on a log-linear scale.  In all cases, it is evident that the nonlinear $\lambda$-saturated frame algorithm outperforms the linear frame algorithm based only on unsaturated coefficients and, in particular, the inclusion of saturated frame coefficients leads to a substantial improvement in the first few iterations, as evidenced by the steeper slope between consecutive iterations.  As expected, the convergence rate for both algorithms improves as $\lambda$ is increased.  As the number of iterations increases, the coefficients $\langle y_{k},x_{j}\rangle$, $j\in J_{\lambda}^{+}(x)\cup J_{\lambda}^{-}(x)$, will approach $\langle x,x_{j}\rangle$ and the benefit of utilizing saturated frame coefficients is diminished.  This is the reason that the linear and nonlinear frame algorithms show the same trend for later iterations. 

\begin{figure}[htp]
\centering
\begin{minipage}[t]{0.45\textwidth}
\vspace{0pt} \centering
\includegraphics[width=0.85\textwidth]{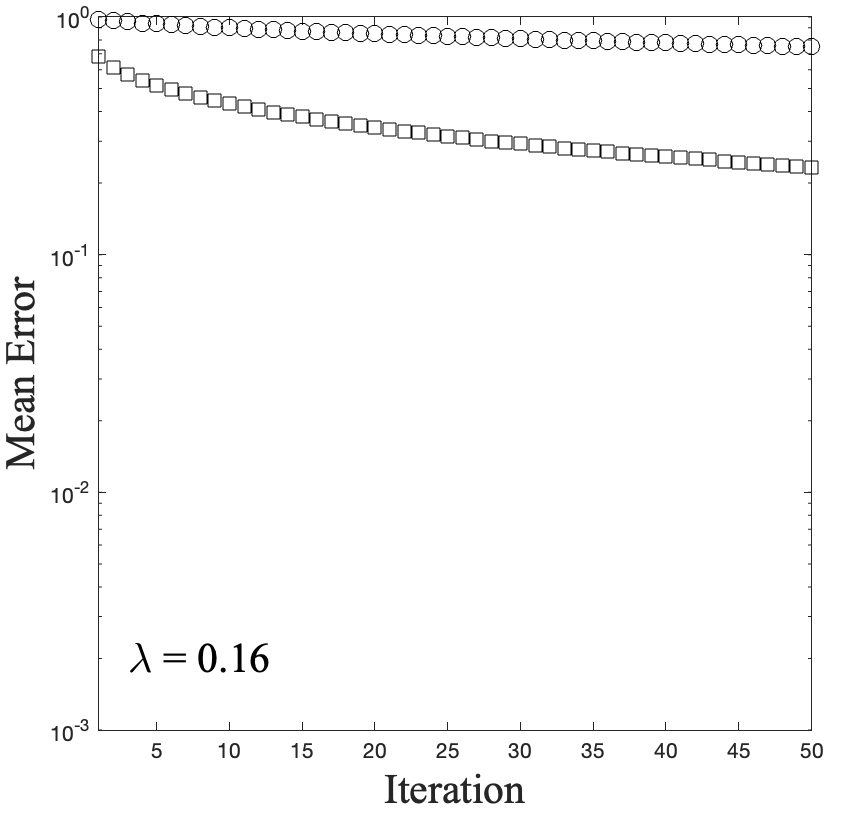}

(a)
\end{minipage} \begin{minipage}[t]{0.45\textwidth}
\vspace{0pt} \centering
\includegraphics[width=0.85\textwidth]{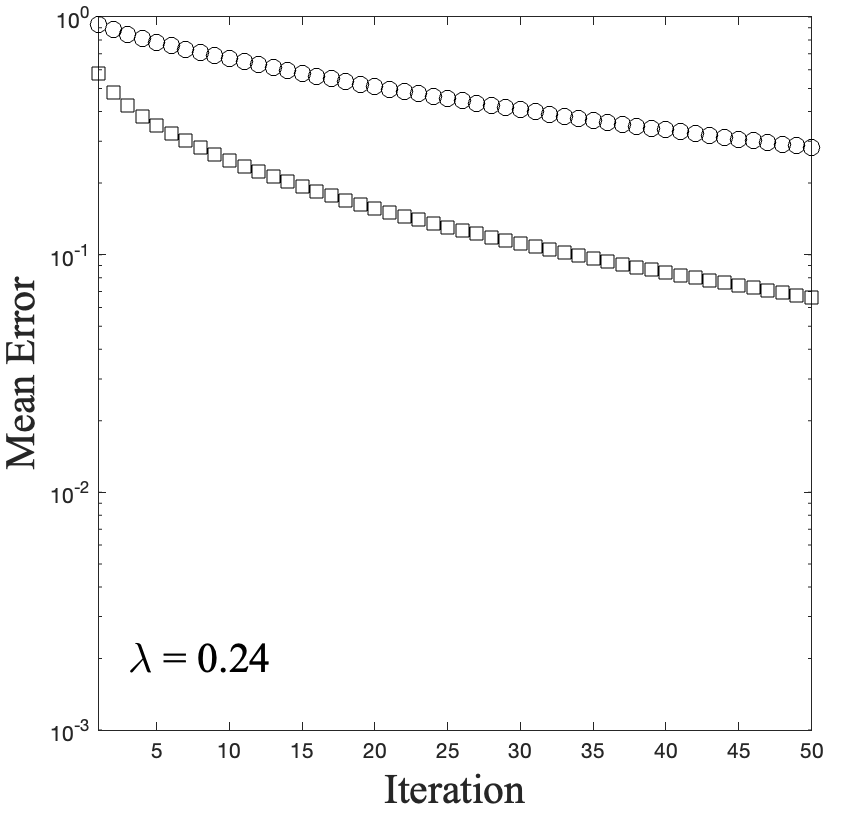}

(b)
\end{minipage}

\bigskip
\begin{minipage}[t]{0.45\textwidth}
\vspace{0pt} \centering
\includegraphics[width=0.85\textwidth]{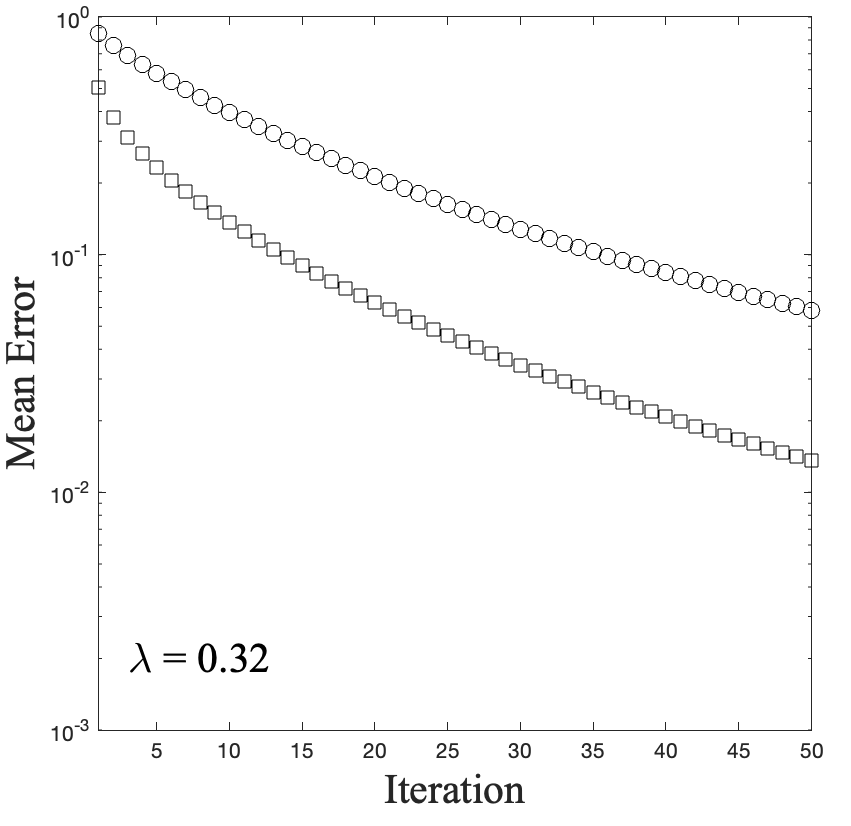}

(c)
\end{minipage} \begin{minipage}[t]{0.45\textwidth}
\vspace{0pt} \centering
\includegraphics[width=0.85\textwidth]{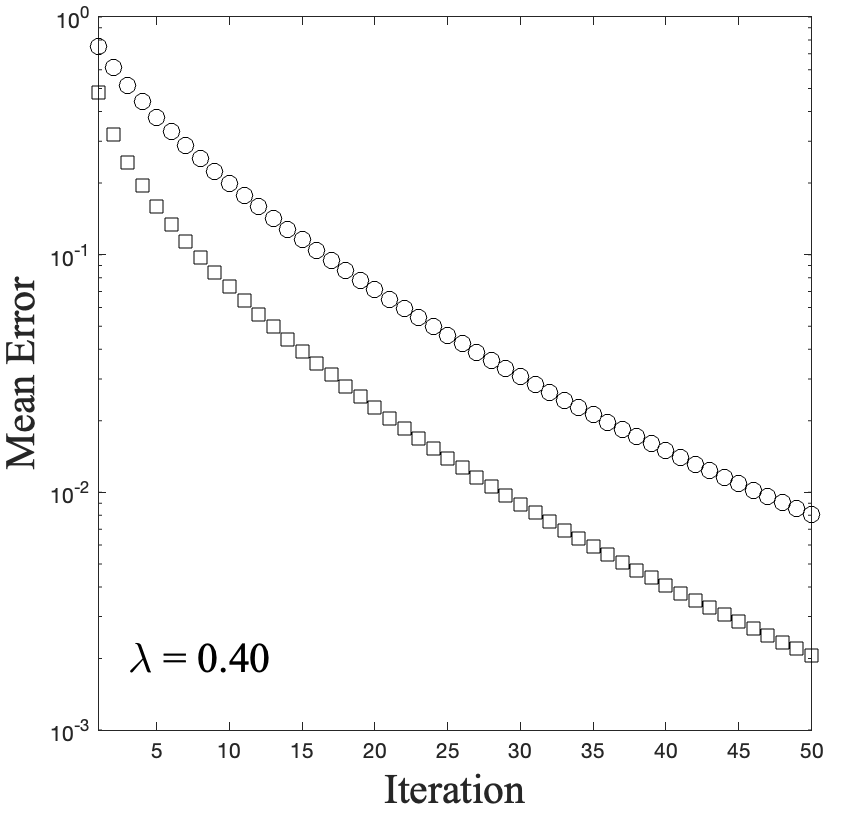}

(d)
\end{minipage}

\caption{Mean error $\Vert y_{k}-x\Vert$ for the linear frame algorithm ($\circ$) and non-linear frame algorithm ($\square$) using random frames of 30 vectors for $\R^{10}$: (a) $\lambda=0.16$ (b) $\lambda=0.24$ (c) $\lambda=0.32$ (d) $\lambda=0.4$.} \label{fig:error}
\end{figure}

Each trial allows for a head-to-head comparison of the two algorithms, which can be quantified by comparing the squared-error on a logarithmic scale.  This quantity, measured in decibels (dB), will be referred to as the \emph{error reduction} and is defined by
$$ \text{Error Reduction} = 10\log_{10}{\left ( \frac{\Vert y_{k} -x\Vert}{\Vert \tilde{y}_{k}-x\Vert} \right )^{2}},$$

\noindent
where $y_{k}$ and $\tilde{y}_{k}$ represent the approximation obtained using the linear and nonlinear frame algorithms after $k$ iterations, respectively.  As $\log_{10}{2} \approx 0.3$, a reduction of the squared-error by a factor of 2 corresponds to a 3 dB error reduction.  Figure \ref{fig:error_red} plots the mean error reduction of the $\lambda$-saturated frame algorithm relative to the linear frame algorithm, including error bars indicating the first and third quartiles from the 1000 trials.  In all four experiments, the mean error reduction after 50 iterations exceeds 9 dB, which shows that, on average, the $\lambda$-saturated frame algorithm reduces the error by a factor of eight in this application.

\begin{figure}[htp]
\centering
\begin{minipage}[t]{0.45\textwidth}
\vspace{0pt} \centering
\includegraphics[width=0.85\textwidth]{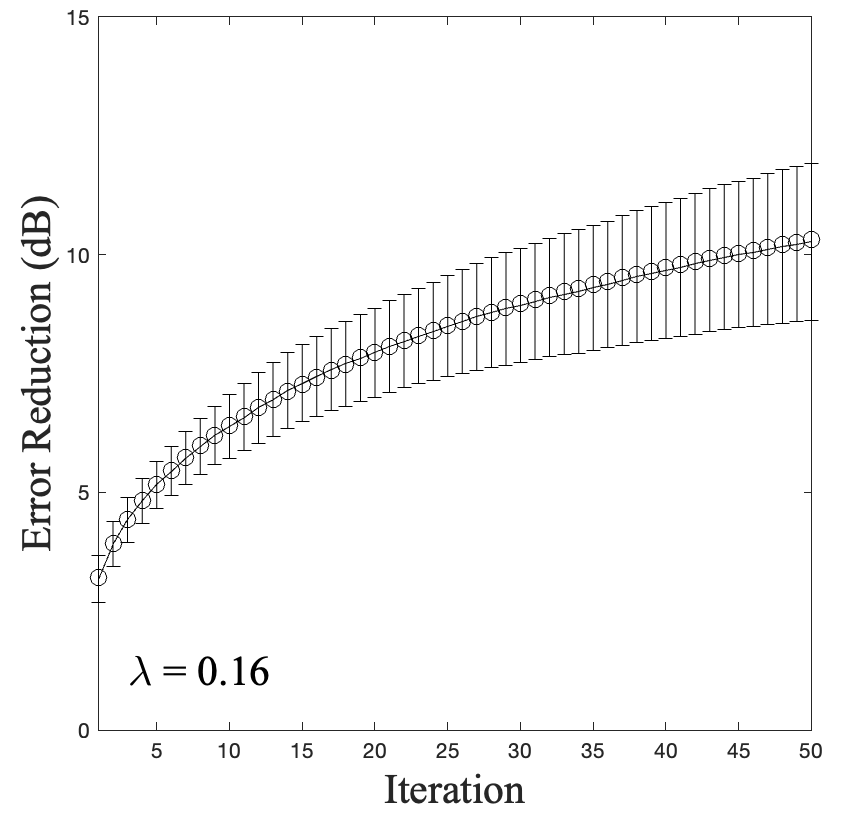}

(a)
\end{minipage} \begin{minipage}[t]{0.45\textwidth}
\vspace{0pt} \centering
\includegraphics[width=0.85\textwidth]{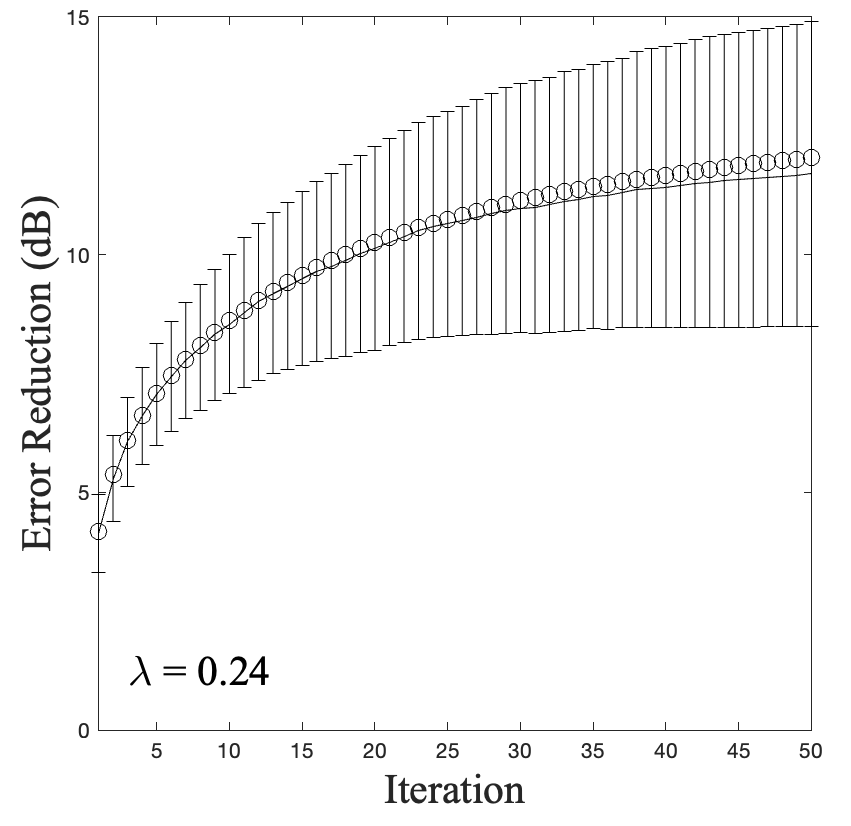}

(b)
\end{minipage}

\bigskip
\begin{minipage}[t]{0.45\textwidth}
\vspace{0pt} \centering
\includegraphics[width=0.85\textwidth]{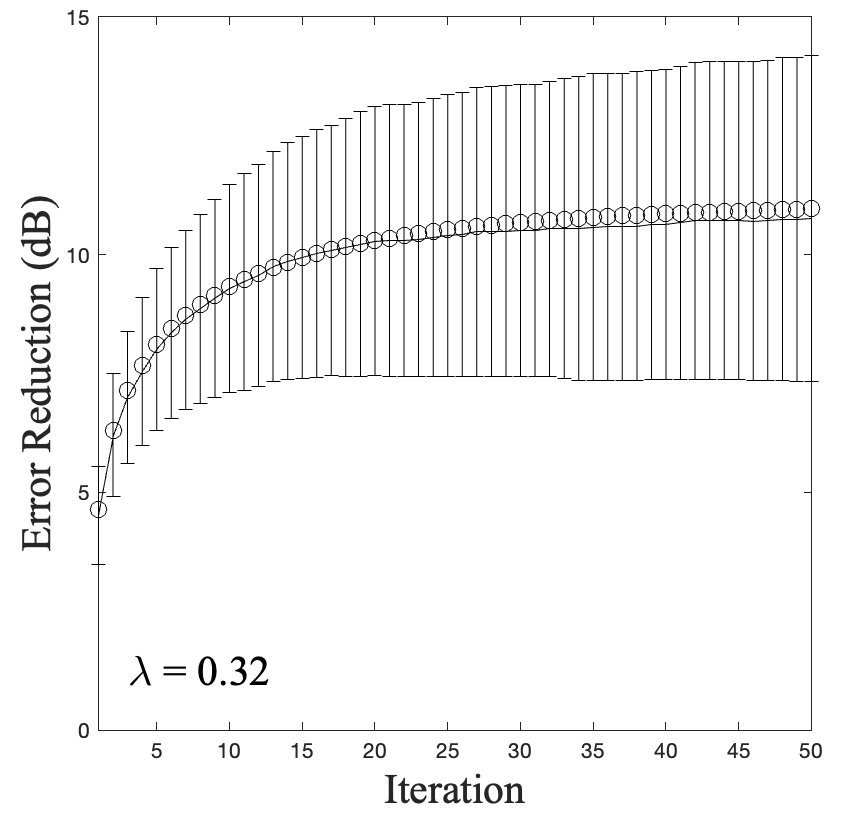}

(c)
\end{minipage} \begin{minipage}[t]{0.45\textwidth}
\vspace{0pt} \centering
\includegraphics[width=0.85\textwidth]{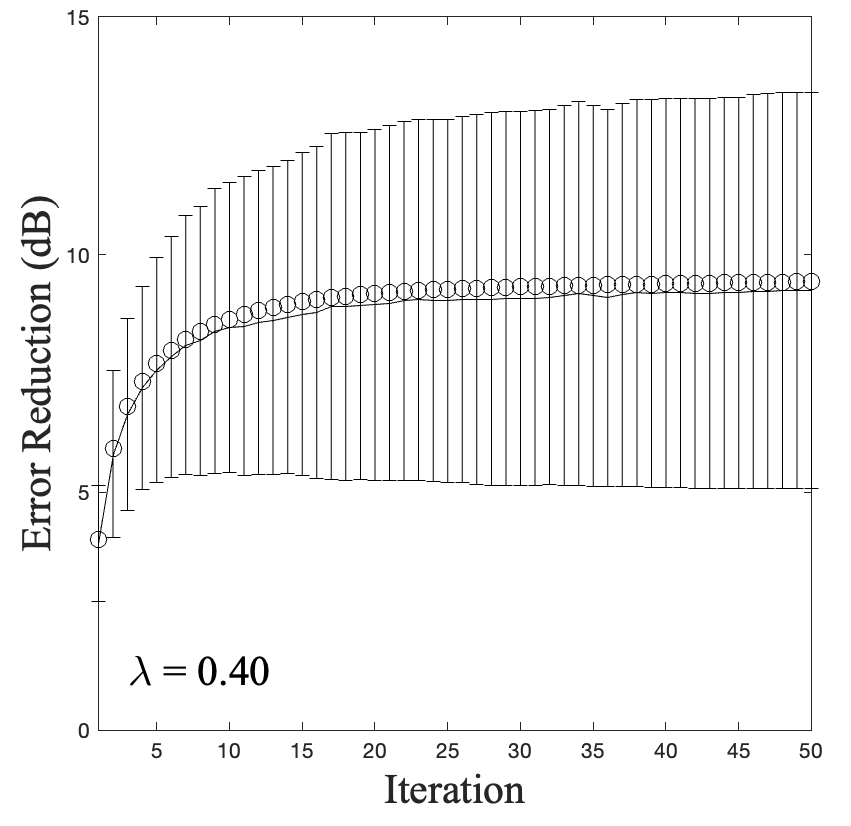}

(d)
\end{minipage}

\caption{Error reduction of the $\lambda$-saturated frame algorithm relative to the linear frame algorithm using random frames of 30 vectors for $\R^{10}$: (a) $\lambda=0.16$ (b) $\lambda=0.24$ (c) $\lambda=0.32$ (d) $\lambda=0.4$.} \label{fig:error_red}
\end{figure}

\end{document}